\documentclass[11pt,letterpaper]{amsart}

\usepackage[no-math]{fontspec}
\usepackage{microtype}
\usepackage{mathtools}
\usepackage{enumitem}
\usepackage{comment}
\usepackage{tikz}
\usepackage{tikz-cd}
\usepackage{amssymb}
\usepackage{hyperref}

\theoremstyle{plain}
\newtheorem{thm}{Theorem}[section]

\newtheorem*{question*}{Question}
\newtheorem*{questions*}{Questions}

\newtheorem{lemma}[thm]{Lemma}
\newtheorem*{thm*}{Theorem}
\newtheorem*{prop*}{Proposition}
\newtheorem{prop}[thm]{Proposition}

\newtheorem{cor}[thm]{Corollary}
\theoremstyle{definition}
\newtheorem{defi}[thm]{Definition}

\theoremstyle{remark}
\newtheorem{rmk}[thm]{Remark}

\newtheorem*{exs}{Example}

\DeclareMathOperator{\coker}{coker}

\DeclareMathOperator{\Br}{Br}

\DeclareMathOperator{\Pic}{Pic}
\DeclareMathOperator{\Hom}{Hom}

\DeclareMathOperator{\Spec}{Spec}

\DeclareMathOperator{\Res}{Res}

\DeclareMathOperator{\im}{Im}

\newcommand{\NS}{\operatorname{NS}}

\newcommand{\Z}{\mathbb{Z}}

\newcommand{\Pp}{\mathbb{P}}
\newcommand{\A}{\mathbb{A}}

\newcommand{\Oo}{\mathcal{O}}

\newcommand{\et}{\textrm{'{e}t}}
\newcommand{\fppf}{\textrm{fppf}}

\newcommand{\Ff}{\mathbb{F}}

\newcommand{\Tr}{\mathrm{Tr}}

\numberwithin{equation}{section}

\title[Differential forms and Brauer classes]
{Differential forms and Brauer classes in positive characteristic}

\author[Domenico Valloni]{Domenico Valloni}

\begin{document}

\begin{abstract}
Let $X$ be a smooth projective variety over a field of characteristic
$p>0$. We introduce an increasing filtration by subgroups $ \Br^\partial_n(X)\subseteq \Br(X)[p]$ of the $p$-torsion Brauer group, obtained from global sections of the absolute K\"ahler
differentials modulo the higher exact differential sheaves $B_{n,X}$.
Over algebraically closed fields, we study this filtration in general and compute it in several cases. For example, if $X$ is a supersingular K3 surface of Artin
invariant $a$, then $\Br^\partial_n(X)=0 \quad \text{for } n<a$ and $\Br^\partial_a(X)=\Br(X).$

Over global function fields, we give an explicit description of the
Brauer--Manin conditions arising from differential Brauer classes.
As applications, we prove $v$-adic non-density results for rational
points on supersingular K3 surfaces and show that, in the presence of
sufficiently many global differential forms, the Brauer--Manin
obstruction can force adelic points to arise from rational points.\end{abstract}

\maketitle

\section{Introduction}
Let $K$ be a field of characteristic $p>0$ and let $X/K$ be a smooth projective variety. There has been growing interest in the $p$-primary torsion of the Brauer group of $X$; see, for example, \cite{MR4685668}, \cite{skorobogatov2024boundednesspprimarytorsionbrauer}, \cite{lazda2025boundednesspprimarytorsionbrauer}, \cite{YANG2026184}, \cite{li2024ptorsionsgeometricbrauergroups}, \cite{grammatica2025brauergroupsabelianvarieties} and \cite{krishna2026brauergroupsvarietieslocal}. It is well known that even the $p$-torsion subgroup of $\Br(X) = H^2_{\et}(X, \mathbb{G}_m)$ can behave in unexpected ways: for instance, if $X/K$ is a supersingular K3 surface and $K$ is infinite, then, after a finite extension of $K$, the group $\Br(X)/\Br(K)$ is an infinite group annihilated by $p$; see \cite[Theorem~1.1 and Remark~1.4]{lazda2025boundednesspprimarytorsionbrauer}. This implies that $\Br(X)/ \Br(K)$ is not finitely generated in general, even when $K$ is finitely generated over its prime subfield. By contrast, for K3 surfaces over number fields and for K3 surfaces of finite height over global fields, $\Br(X)/\Br(K)$ is always finite; see \cite{MR2395136} and \cite{lazda2025boundednesspprimarytorsionbrauer}. This raises the question, posed to the author by Alexei Skorobogatov, of whether such \(p\)-torsion classes contribute to the Brauer--Manin obstruction to the Hasse principle.

The paper has two main parts. First, we introduce and study a
filtration of the $p$-torsion Brauer group by differential Brauer
classes, and we show that this filtration recovers the Brauer group of supersingular K3 surfaces. Second, over global function fields, we give an explicit description of the Brauer--Manin pairing with these classes, which we then use to answer the question above.

We begin with the construction of the filtration. Denote by $\Omega^1_X \coloneqq \Omega^1_{X/\Ff_p}$ the sheaf of absolute K\"{a}hler differentials of $X$, and by $B_{1,X} \subset \Omega^1_X$ the subsheaf of exact differentials. The subsheaves $B_{n,X} \subset \Omega^1_X$ are then defined inductively (see Section \ref{first section}), with the convention $B_{0,X} = 0$. Each $B_{n,X}$ is a quasi-coherent sheaf of $\Oo_X^{p^n}$-modules (coherent if $K$ is $F$-finite, i.e., if $[K : K^p] < \infty$). In Section \ref{Section differential Brauer classes}, we show that for any $n \geq 0$ there is a natural homomorphism (see also, for example, \cite[Proposition 2.2.9]{MR0565469})
\[
H^0(X, \Omega^1_X / B_{n,X}) \;\longrightarrow\; \Br(X)[p], \quad \omega \mapsto [\omega].
\]  
We denote by $\Br^{\partial}_{n}(X) \subset \Br(X)[p]$ the image of this map, and we call it the subgroup of \emph{differential Brauer classes} of level $n$. We note that maps of this form for $n=0,1$ have appeared in the literature, although in contexts different from the one considered here; see, for example,\cite[Proposition 4.2]{zbMATH05277424} and \cite[Section 3.6]{zbMATH06591241}. The observation we use is that the construction extends to the quotient sheaves
$\Omega^1_X/B_{n,X}$ for every $n\geq 0$, thus yielding an
increasing filtration 
\[
\Br^{\partial}_{0}(X) \subset \Br^{\partial}_{1}(X) \subset \cdots \subset \Br^{\partial}_{n}(X) \subset \cdots \subset \Br(X)[p].
\]  
We compute this filtration over algebraically closed fields in the first part of the article; see Theorem \ref{Main thm filtration}. For example, we use the computations in \cite{MR1776939} to show that this construction recovers the Brauer group of supersingular K3 surfaces: if $X$ is a supersingular K3 surface of Artin invariant $a$ over an algebraically closed field $K$, then $\Br^{\partial}_{n}(X) = 0$ for $n < a$ and $\Br^{\partial}_{a}(X) = \Br(X)$. To our knowledge, this description of the Brauer group of supersingular K3 surfaces in terms of differential forms, with the Artin invariant determining the first nonzero step of the filtration, has not previously appeared in the literature. This already has nontrivial consequences for the distribution of global points on supersingular K3 surfaces; see Theorem \ref{Theorem Intro K3} and Section \ref{sect Constant}. 

Next, we show that the Brauer--Manin pairing with differential Brauer classes admits an explicit description in terms of differential forms. Let $K$ be a global field of characteristic $p>0$ and let $\A_K$ be its adele ring. For each place $v$ of $K$, we write $K_v$ for the completion of $K$ at $v$. Given a global form $\omega \in H^0(X, \Omega^1_X / B_{n,X})$ and a local point $x_v \in X(K_v)$, we define $ \omega|_{x_v} \in \Omega^1_{K_v} / B_{n,K_v} $ to be the pullback of $\omega$ via $x_v \colon \Spec(K_v) \to X$.
\begin{thm}\label{theorem 1 intro}
Let $(x_v)\in X(\A_K)$ and let
$\omega\in H^0(X,\Omega^1_X/B_{n,X})$. Then $(x_v)$ is orthogonal to
every Brauer class of the form
\[
[t^{p^n}\omega],\qquad t\in K,
\]
if and only if
\[
C^n\bigl((\omega|_{x_v})_v\bigr)
\in \Omega^1_K\subseteq\Omega^1_{\A_K}.
\]
\end{thm}
Here $C \colon \Omega^1_{\A_K} \to \Omega^1_{\A_K}$ is the Cartier operator (see Section \ref{first section} for precise definitions and properties). When $X$ is constant, the Brauer--Manin obstruction coming from $\Br^\partial_0(X)$ is intimately related to the $F$-descent obstruction coming from the Albanese map\footnote{As pointed out to us by Felipe Voloch}. On the other hand, the obstruction coming from $\Br^\partial_0(X)$ is strictly stronger than the $F$-descent obstruction whenever the dimension of the Albanese variety of $X$ is strictly smaller than $h^0(X, \Omega^1_X)$ (e.g., Igusa's example). This means that the theorem above for $n=0$ is most useful when studying the Brauer--Manin obstruction for \emph{pathological} varieties in positive characteristic, as we explain in Section \ref{F-descent sub} (see also \cite{bernasconi2026brauermaninobstructionfailurekodaira} for a recent application). 
\subsection{Applications}
The results above have several consequences for the structure of the Brauer--Manin set of varieties admitting differential Brauer classes. Let $\omega\in H^0(X,\Omega^1_X/B_{n,X}) $ be a nonzero global section. For every $t\in K$, it determines a
Brauer class $[t^{p^n}\omega]\in\Br(X)[p]. $ Thus, for each place $v$ of $K$, we obtain a family of ordinary
evaluation maps
\[
\mathrm{ev}_{[t^{p^n}\omega],v}\colon
X(K_v)\longrightarrow\Br(K_v),
\qquad t\in K.
\]
Theorem \ref{theorem 1 intro} allows us to study this entire family
simultaneously through the single pullback map
\[
\begin{aligned}
\mathrm{ev}^{\partial}_{\omega,v}\colon
X(K_v)&\longrightarrow\Omega^1_{K_v}/B_{n,K_v},\\
x_v&\longmapsto\omega|_{x_v}
\end{aligned}
\]
which is considerably easier to study than the individual Brauer
evaluation maps. Evaluation maps associated with Brauer classes are currently the
subject of active research; see
\cite{zbMATH06154640,zbMATH07748318,zbMATH07590357,
ambrosi2025wildbrauerclassesprismatic}.
These works focus primarily on $p$-adic fields. Recently,
\cite{krishna2026brauergroupsvarietieslocal} extended the results of
\cite{zbMATH07748318} to positive characteristic. In Section~\ref{sec evaluation map}, we introduce $p^n$-analytic maps
and show that $\mathrm{ev}^{\partial}_{\omega,v}$ is $p^n$-analytic.
Under the natural Frobenius-descent hypotheses of
Proposition~\ref{main prop evaluation map 2}, we also prove that this
map is nonconstant on every neighbourhood of every point of $X(K_v)$.
Applying these results to K3 surfaces, we obtain:
\begin{thm}\label{Theorem Intro K3}
Let $X/K$ be a supersingular K3 surface of Artin invariant $a$, and
assume that $X$ descends to $K^{p^a}$ (this hypothesis can always be
arranged by replacing $K$ with $K^{1/p^a}$ and $X$ with
$X_{K^{1/p^a}}$). Then

\begin{enumerate}
\item
The Brauer--Manin set $X(\A_K)^{\Br(X)[p]}$ is contained in a finite
union of subsets of the form
\[
\prod_v W_v,
\]
where each $W_v\subseteq X(K_v)$ is closed and has $v$-adic measure
zero. In particular, the $v$-adic closure of $X(K)$ in $X(K_v)$ has
measure zero for every place $v$.

\item
Let $X_a$ denote the Weil restriction of $X$ from $K$ to $K^{p^a}$.
Under the natural identification
\[
X_a(K_v^{p^a})=X(K_v),
\]
the sets $W_v$ may be chosen to be proper analytic subvarieties.
\end{enumerate}
\end{thm}
This contrasts with both finite-height K3 surfaces over global fields
and K3 surfaces over number fields: in these cases,
$\Br(X)/\Br(K)$ is finite, so $X(\A_K)^{\Br(X)}$ is open in the adelic
topology. If this set is nonempty, Skorobogatov's conjecture~\cite{skorobogatov2009diagonalquartics}\footnote{This conjecture asserts
that, for a K3 surface over a number field, the closure of $X(K)$ in $X(\A_K)$ is
$X(\A_K)^{\Br(X)}$.} (or its global function field analogue) would therefore imply, in the finite-height and
characteristic-zero cases, that the closure of $X(K)$ in $X(K_v)$
equals $X(K_v)$ for almost every place $v$. By contrast, the theorem
above shows that for supersingular K3 surfaces the rational points are $v$-adically sparse at every place, and part (2) refines this statement by characterizing their closure in terms of analytic subvarieties (after Weil restriction). 

In the final part of the article we apply Theorem \ref{theorem 1 intro} in the case $n=0$ to obtain more precise results. The main theme is that the stronger the positivity properties of $\Omega^1_{X/K}$ are, the stronger the Brauer--Manin obstruction
arising from $\Br^\partial_0(X)$ becomes. For example, let $\mathcal{L} \subset \Omega^1_{X}$ be a globally generated invertible subsheaf, and let
\[
\psi_{\mathcal{L}} \colon X \longrightarrow \Pp(H^0(X, \mathcal{L})^\vee) \cong \Pp^N_K
\]
be the induced $K$-morphism.
\begin{thm} \label{Theorem Intro LinBun}
Let $(x_v) \in X(\A_K)^{\Br(X)[p]}$ and assume that there exists a local component $x_v$ and a differential form $\omega \in H^0(X, \mathcal{L})$ such that $0 \neq \omega|_{x_v} \in \Omega^1_{K_v}$. Then
\[
\psi_{\mathcal{L}}((x_v)) \in \Pp^N(K).
\]
In particular, if $\mathcal{L}$ separates points, then $(x_v) \in X(K)$.
\end{thm}
Since Bogomolov--Sommese vanishing---which implies that the Kodaira
dimension of an invertible subsheaf
$\mathcal L\subset\Omega^1_{X/K}$ is at most one---can fail in positive characteristic (see, for instance, \cite{MR4473638}), the theorem applies beyond curves, even when $\mathcal L$ is ample. In the recent preprint \cite{bernasconi2026brauermaninobstructionfailurekodaira}, we apply
the result to Raynaud surfaces. We conclude by studying the
Brauer--Manin set for smooth projective varieties whose
$\Omega^1_{X/K}$ is globally generated; see
Theorem~\ref{Theorem glob gen} and the surrounding discussion. For
instance, when $\Omega^1_{X/K}$ is $1$-jet ample, we show that
\[
X^{\mathrm{good}}(\mathbb A_K)^{\Br^\partial_0(X)}
=
X^{\mathrm{good}}(K),
\]
where $X^{\mathrm{good}}(\mathbb A_K)$ denotes the set of ``good
adelic points,'' with $X^{\mathrm{good}}(K)$ defined analogously; see
Definition~\ref{def good points}.
\subsection{Literature on the Brauer-Manin set in characteristic \texorpdfstring{$p>0$}{p > 0}}
The Brauer--Manin obstruction over global function fields has been studied in several works, mainly for curves and more generally for subvarieties of abelian varieties. A first general result is due to Poonen and Voloch \cite{MR2630046}. Their Theorem D proves that $X(\A_K)^{\Br(X)}=X(K)$ for a coset-free subvariety $X$ of an abelian variety $A$, under the assumptions that $A$ has no isotrivial factors and that $A(K^s)[p^\infty]$ is finite. The latter condition had already appeared in \cite{zbMATH00811732}, where it is shown to hold generically, for example when $A$ is ordinary and the Kodaira--Spencer map has full rank. Creutz \cite{creutz2025adelicmordelllangbrauermaninobstruction} recently proved Conjecture C of \cite{MR2630046}, thus removing these additional assumptions, except for the assumption that $A$ has no isotrivial quotients.

The case of curves has also been treated separately. For constant curves, \cite{MR4448590} describes the complement of the rational points in the Brauer--Manin set, and shows that in certain cases the equality $C(K)=C(\mathbb{A}_K)^{\mathrm{Br}(C)}$ holds. On the other hand, for non-isotrivial curves, the equality $C(K)=C(\mathbb{A}_K)^{\mathrm{Br}(C)}$ always holds; see \cite{creutz2023brauermaninobstructionnonisotrivialcurves}.
\section{Differential Brauer classes}
In this chapter we study the $p$-torsion Brauer group of a regular scheme of positive characteristic. Most of the results presented here are classical constructions in positive characteristic. Some of the references that we use are Illusie's \cite{MR0565469}, Milne's \cite{MR559531}, Serre's \cite{MR98097} and Gille's and Szamuely's book \cite{MR3727161}
\subsection{Preliminaries} 
\label{first section}
The inverse Cartier operator was introduced in \cite{MR0084497} and it can be defined over any base scheme of positive characteristic, but for our applications, it is useful to consider only the absolute case, i.e., with $\Spec(\Ff_p)$ as base scheme. Let $A$ be a reduced $\Ff_p$-algebra. We denote by $\Omega^1_A$ the module of K\"{a}hler differentials of $A$ over $\Ff_p$ and by $A^p$ the subalgebra of $A$ generated by $p$-powers. Then, the de Rham complex $(\Omega^\bullet_A, d)$ is $A^p$-linear (where $A^p \subset A$ is the subring of $p$-powers) and it follows that $\mathcal{H}^i(\Omega^\bullet_A)$ is a quasi-coherent $A^p$-module (coherent if $A^p \subset A$ is finite). The inverse Cartier operator is then the unique map of graded $A$-algebras
\begin{equation}
    C^{-1} \colon \bigoplus_{i \geq 0 }\Omega^i_A \rightarrow \bigoplus_{i \geq 0 } H^i(\Omega^\bullet_A) 
\end{equation}
(where the $A$-algebra structure on the right-hand side is given by $a \mapsto a^p$) which satisfies the following properties:
\begin{enumerate}[label=(\roman*)]
    \item $C^{-1}(a) = a^p$ for $a \in A = \Omega^0_A$;
    \item $C^{-1} (da) = [a^{p-1} da]$ for $a \in A$;
    \item $C^{-1}$ respects the graded algebra structure on both sides. 
\end{enumerate}
Cartier then proved the following:
\begin{thm}
If $A$ has a $p$-basis, then $C^{-1}$ is an isomorphism.
\end{thm}
We recall that a ring $A$ of characteristic $p>0$ has a $p$-basis if there is a subset $\Gamma \subset A$ such that $A^p[\Gamma] = A$ and, for any finite choice of elements $a_1 \cdots a_k \in \Gamma$, the monomials $a_1^{i_1} \cdots a_k^{i_k}$ with $0 \leq i_1, \cdots, i_k \leq p-1$ are linearly independent over $A^p$ (see, for instance, an equivalent statement in \cite{MR252389} or \cite{MR0567425}). By \cite[\href{https://stacks.math.columbia.edu/tag/07P2}{Tag 07P2}]{stacks-project}, any field $K$ of positive characteristic has a $p$-basis. Moreover, \cite[Theorem~3.4]{MR0567425} shows that the local rings of a smooth algebra over a field of characteristic $p>0$ have $p$-bases. We use these constructions locally and then sheafify.   
In this case, one defines the Cartier operator $C$ using the short exact sequence 
\begin{equation} \label{sequence: Cartier}
  0 \rightarrow  B^1_A \rightarrow Z^1_A \xrightarrow{C} \Omega^1_A \rightarrow 0,
\end{equation}
where $B^1_A \subset Z^1_A \subset \Omega^1_A$ denote, respectively, the exact and the closed forms of $\Omega^1_A$ (note that these are $A^p$-modules). 
\begin{cor} \label{cor: d = 0 iff p power}
If $A$ has a $p$-basis, then $d(a) = 0$ if and only if $a \in A^p$.
\end{cor}
Let now $\Omega_{\mathrm{log},A}^1 \subset Z^1_A$ be the abelian subgroup generated by the forms $\mathrm{dlog}(a) = da/a$ for $a \in A^\times$. Note that $a^{-1}da = a^{-p}a^{p-1} da$ and hence $$C( \mathrm{dlog}a) = C(a^{-p}a^{p-1} da) = a^{-1} C(a^{p-1} da) = a^{-1} da = \mathrm{dlog}(a).$$ The following result is also due to Cartier \cite{MR0084497}:
\begin{thm}
 If $A$ has a $p$-basis, then $w\in Z^1_A$ satisfies $C(w)=w$ if and only if it is Zariski-locally of the form $da/a$ for a unit $a$.
\end{thm}
The next result characterises logarithmic forms. We let $U = \Spec(A)$. 
\begin{prop} \label{prop: ses log}
If $A$ is smooth over a field $K$, there is a short exact sequence on $U_{\et}$ 
$$0 \rightarrow \Omega_{U,\mathrm{log}}^1 \rightarrow Z^1_U \xrightarrow{1-C} \Omega^1_U \rightarrow 0.$$
\end{prop}
For a proof, see for instance \cite{MR0565469}[Corollary 2.1.18]. When $A$ has a $p$-basis, one defines the abelian sheaves $B^1_{A,n} \eqcolon B_n$ iteratively as follows: $B_0 = 0$ and $B_{n+1} = \{ \omega \in Z^1_A \colon C(\omega) \in B_{n} \}$. So $B_1 = B_A^1$ by the previous considerations. Similarly, one defines the sheaves $Z_n$. Note that $B_n$ and $Z_n$ are naturally sheaves of $A^{p^n}$-modules. One obtains a chain of inclusions 
$$0 \subset B^1_A = B_1 \overset{\overset{C}{\curvearrowleft}}{\subset} B_2 \overset{\overset{C}{\curvearrowleft}}{\subset} B_3 \overset{\overset{C}{\curvearrowleft}}{\subset} \cdots B_\infty \subset Z_\infty \subset \cdots \overset{\overset{C}{\curvearrowright}}{\subset} Z_3 \overset{\overset{C}{\curvearrowright}}{\subset} Z_2 \overset{\overset{C}{\curvearrowright}}{\subset} Z_1 = Z^1_A$$
and one has short exact sequences 
$$0 \rightarrow{B^1_A} \rightarrow B_{n+1} \xrightarrow{C} B_n \rightarrow 0 $$
and 
$$0 \rightarrow{B^1_A} \rightarrow Z_{n+1} \xrightarrow{C} Z_n \rightarrow 0.$$
Note that $C$ acts as a nilpotent operator on $B_\infty$. Hence $B_\infty \cap \Omega^1_{A, \log} = 0$ necessarily. Cartier \cite{MR0084497} and Serre \cite{MR98097} gave a description of the sheaves $B_n$ using Witt vectors (still assuming that $A$ has a $p$-basis). Let $W_n(A)$ be the truncated Witt vector of length $n$ of $A$. One defines differentials $D_n \colon W_n(A) \rightarrow \Omega^1_A$ by the rule 
$$\underline{f} = (f_0,f_1, \cdots, f_{n-1}) \mapsto  d f_{n-1} + f_{n-2}^{p-1} df_{n-2} + \cdots + f_0^{p^{n-1}-1}d f_0.$$
\begin{thm}[Cartier, Serre] 
The map $D_n$ is additive and satisfies the twisted Leibniz rule ${D_n(\underline{f} \cdot \underline{g}) = D_n(\underline{f})g_0^{p^{n-1}} + f_0^{p^{n-1}}D_n(\underline{g}).}$ The image of $D_n$ is $B_n$, and its kernel is $F W_n$.
\end{thm}
\begin{proof}
We only prove that the image is $B_n$; the rest of the statements can be found in the cited references. We prove this by induction on $n$. If $n=1$, this is \eqref{sequence: Cartier}. Now, assume that the statement is true for $n-1$. Let $\omega \in B_n$, which means that $C^n(\omega) = 0$. Now $C(C^{n-1} \omega) = 0$, hence $C^{n-1} \omega = d f$ for some $f \in A$. Note also that $d f = C^{n-1} (f^{p^{n-1}-1} df)$, so $C^{n-1}( \omega - f^{p^{n-1}-1} df) = 0$. In particular, $\omega - f^{p^{n-1}-1} df \in B_{n-1}$, and by induction we can find a Witt vector of length $n-1$ such that $D_{n-1} ((f_0,f_1, \cdots, f_{n-1})) = \omega - f^{p^{n-1}-1} df$. Consider now the Witt vector of length $n$ given by $(f, f_0,f_1, \cdots, f_{n-1})$ and compute 
$$D_n ((f, f_0,f_1, \cdots, f_{n-1})) = D_{n-1}(f_0,f_1, \cdots, f_{n-1}) + f^{p^{n-1}-1}df = \omega.$$
\end{proof}
Note that all these constructions can be sheafified on the Zariski site. 
\subsection{The \texorpdfstring{$p$}{p}-torsion Brauer group} \label{section p torsion}
Let now $X$ be a regular scheme. Then the Zariski presheaf $ X \supset U \mapsto \Br(U)[p]$ is actually a sheaf by \cite{MR4304038}[Corollary 3.5.6, Theorem 3.5.7]. We denote this sheaf by $\Br_X[p]$. 

\begin{thm} \label{thm p torsion br}
We have a natural exact sequence on $X_{\mathrm{Zar}}$:
$$0 \rightarrow \Omega^1_{\log, X} \rightarrow Z^1_X \xrightarrow{1-C} \Omega^1_X \rightarrow \Br_X[p] \rightarrow 0.$$
\end{thm}

\begin{proof}
Consider the sheaf $\mu_p$ on $X_{\mathrm{fppf}}$. Then, for any regular scheme $X$, Kummer theory and Grothendieck's version of Hilbert's 90 give an exact sequence 
$$0 \rightarrow \Pic(X)/ p \Pic(X) \rightarrow H^2(X_{\mathrm{fppf}}, \mu_p) \rightarrow \Br(X)[p] \rightarrow 0.$$
On the small Zariski site of $X$, consider the sheafification of $U \mapsto H^2(U_{\mathrm{fppf}}, \mu_p)$. Since every class of $\Pic(U)/p \Pic(U)$ can be killed by passing to smaller open subsets, this sheaf must correspond to $\Br_X[p]$. Now, Milne's isomorphism $H^{2}(U_{\mathrm{fppf}}, \mu_p) \cong H^1(U_{\et}, \Omega^1_{X, \log})$ allows us to compute $H^2(X_{\mathrm{fppf}}, \mu_p) $ in a different way. Before doing so, we give a quick proof of Milne's isomorphism for the convenience of the reader. By \cite{MR559531}[Theorem 3.9], if $\pi \colon U_{\mathrm{fppf}} \rightarrow U_{\et}$ is the restriction of topologies, we have $R^i \pi_* \mathbb{G}_m = 0$ for $i>0$ and $ \pi_* \mathbb{G}_m = \mathbb{G}_m $. Now, we have a short exact sequence of sheaves on $U_{\mathrm{fppf}}$:
$$0 \rightarrow \mu_p \rightarrow \Oo_U^\times\xrightarrow{(-)^p} \Oo_U^\times \rightarrow 0 $$
from which we deduce the short exact sequence on $U_\et$:
$$0 \rightarrow \Oo_U^\times \xrightarrow{(-)^p} \Oo_U^\times \rightarrow R^1 \pi_* \mu_p \rightarrow 0 $$
and hence the isomorphism $R^1 \pi_* \mu_p \cong \Omega^1_{U, \log}$. Milne's isomorphism then follows from the Leray spectral sequence for $\pi_*$ applied to $\mathbb{G}_m$. 

Now Proposition \ref{prop: ses log} gives the long exact sequence 
$$0 \rightarrow H^0(U, \Omega^1_{\log, U}) \rightarrow H^0(U,Z^1_U) \xrightarrow{1-C} H^0(U,\Omega^1_U) \rightarrow H^1(U_{\et}, \Omega^1_{\mathrm{log}, U}) \rightarrow H^1(U_{\et}, Z^1_U);$$
but $Z^1_U$ is a (quasi-)coherent sheaf of $\Oo_U^p$-modules (coherent as soon as $K^p \subset K$ is finite), so $ H^1(U_{\et}, Z^1_U) =  H^1(U_{\mathrm{Zar} }, Z^1_U) = 0$ 
whenever $U$ is affine. This proves the result.  
\end{proof}
The map $\Omega^1_X \rightarrow \Br_X[p]$ can be described explicitly in terms of Azumaya algebras. We do this only when $X = \mathrm{Spec}(K)$, following the computations found in \cite[Section 9.2]{MR3727161}. When $X$ is regular and integral and $K$ is its function field, the Azumaya algebra obtained in $\Br(K)$ will extend to the whole $\Br(X)$. Let $\chi \in \Hom(G_K, \Ff_p) \cong H^1(G_K, \Ff_p)$ be a continuous character and let $g \in K^\times = H^0(G_K, (K^s)^\times)$. Finally, let $\partial \colon H^1(G_K, \Ff_p) \xrightarrow{\sim} H^2(G_K, \Z)$ be the natural isomorphism given by the boundary map. The cup product $\partial(\chi) \cup g$ defines a Brauer class of $K$ denoted by $[\chi, g)$. If $f \in K$ and $\chi$ is the character associated to the Artin-Schreier equation $x^p-x = f$, such a Brauer class corresponds to the cyclic algebra $$[f,g) = \langle x,y \colon x^p-x = f ; y^p = g ; xy=y(x+1) \rangle.$$
\begin{prop} \label{prop azu}
Let $K$ be a field and let $f,g \in K^\times$. Then, the map from Theorem \ref{thm p torsion br} sends $f d(g) \in \Omega^1_K$ to the Azumaya algebra $[fg,g)$.
\end{prop}

\begin{proof}
Let $K^s$ be a separable closure of $K$ and consider the short exact sequence of $G_K$-modules 
$$ 0 \rightarrow \Omega^1_{\log, K^s} \rightarrow Z^1_{K^s} \xrightarrow{1-C} \Omega^1_{K^s} \rightarrow 0;$$
now pick $h \in (K^s)^{\times}$ such that $h^p - h = fg$ and consider the closed form $z = h^p \mathrm{dlog}(g)$; then $(1-C)(z) = f d(g)$ and so $[f d(g)] \in H^1(G_K, \Omega^1_{\log, K^s})$ is represented by the cocycle $\sigma \mapsto (\sigma(h)-h) \mathrm{dlog}(g)$. But then $$[f d(g)] = \chi_h \cup \mathrm{dlog}(g) $$ where $\chi_{h} \colon G_{K^s} \rightarrow \Z / p \Z$ is the character coming from the Artin-Schreier equation $x^p - x = fg$ and $\mathrm{dlog}(g) \in H^0(X, \Omega^1_{\log, K^s})$. But by Hilbert's 90 we have $H^0(X, \Omega^1_{\log, K^s}) = K^\times/ K^{\times p} = \Omega^1_{\log,K}$, so the cocycle above is $[fd(g)] = \chi_h \cup g$ where $g \in K^\times = H^0(G_K, (K^s)^\times)$. Finally, let $\partial' \colon H^1(G_K, \Omega^1_{\log, K^s}) \xrightarrow{\sim} \Br(K)[p]$ be the isomorphism induced by the boundary map. By \cite[Proposition 3.4.8]{MR3727161} we have the identity $\partial(\chi_h) \cup g = \partial'( \chi_h \cup g) = [fg,g),$ which concludes the proof.
\end{proof}

\subsection{Example: local fields}
\label{section: K_q((t))}
Let us verify that ${\Br(\Ff_q((t)))[p] \cong \Z/p\Z}$, which follows from local class field theory. This example will be used to prove Theorem \ref{theorem 1 intro}. We begin by noticing that $\Ff_q((t))$ always admits the $p$-basis $\Gamma = \{ t \}$ so all the results from the previous sections apply. Now, note that $\Omega^1_{\Ff_q((t))} \cong \Ff_q((t)) dt.$ The analogous statement is false in characteristic zero without a continuity assumption, since otherwise $\Omega^1_{\mathbb C((t))/\mathbb C}$ has uncountable dimension over $\mathbb C((t))$.

\begin{prop}
We have a short exact sequence of abelian groups
$$0 \rightarrow (1-C)\Omega^1_{\Ff_q((t))} \rightarrow \Omega^1_{\Ff_q((t))} \xrightarrow{\Tr_{\Ff_q/ \Ff_p} \circ \Res} \Ff_p \rightarrow 0,$$ where $ \Res \colon \Omega^1_{\Ff_q((t))} \rightarrow \Ff_q$ is the usual residue of a differential form. In particular, $\Br(\Ff_q((t)))[p] \cong \Ff_p$.
\end{prop}
\begin{proof}
 Every $\omega\in\Omega^1_{\Ff_q((t))}$ can be written uniquely as
$\omega=\sum_{m\in\mathbb Z}a_m t^m\mathrm{dlog}(t)$, where
$a_m\in\Ff_q$ and $a_m=0$ for $m\ll 0$. In this notation,
$\Res(\omega)=a_0$ and
$C(\omega)=\sum_{m\in\mathbb Z}a_{pm}^{1/p}t^m\mathrm{dlog}(t)$.
Since the trace is invariant under Frobenius, we have
$\Tr_{\Ff_q/\Ff_p}(\Res((1-C)\omega))=0$. Hence
$(1-C)\Omega^1_{\Ff_q((t))}$ is contained in the kernel of
$\Tr_{\Ff_q/\Ff_p}\circ\Res$. Moreover, this last map is surjective, because both the residue map and the trace map are surjective. Conversely, write
$\omega=\sum_{m\in\mathbb Z}a_m t^m\mathrm{dlog}(t)$ and assume that
$\Tr_{\Ff_q/\Ff_p}(a_0)=0$. We look for
$\widetilde{\omega}=\sum_{m\in\mathbb Z}b_m
t^m\mathrm{dlog}(t)$ such that
$(1-C)(\widetilde{\omega})=\omega$. Comparing coefficients, we need
$b_m-b_{pm}^{1/p}=a_m$ for every $m$. For $m=0$, this equation has a solution precisely because
$\Tr_{\Ff_q/\Ff_p}(a_0)=0$. Indeed, the image of
$b\mapsto b-b^{1/p}$ is the kernel of the trace. For the remaining
coefficients, fix an integer $m$ not divisible by $p$ and consider the coefficients with indices $m,pm,p^2m,\ldots$. If $m>0$, choose $b_m$ and use the equations successively to determine
$b_{pm},b_{p^2m},\ldots$. If $m<0$, then
$a_{p^rm}=0$ for $r$ sufficiently large; set $b_{p^rm}=0$ for such
$r$ and solve the remaining equations backwards. In this way we obtain a Laurent series $\widetilde{\omega}$ satisfying
$(1-C)(\widetilde{\omega})=\omega$.
\end{proof}
\subsection{Differential Brauer group} \label{Section differential Brauer classes}
Let now $X/K$ be a smooth variety. Theorem \ref{thm p torsion br} yields a natural map $H^0(X, \Omega^1_X) \rightarrow \Br(X)[p]$. To obtain such maps also for varieties without global differential forms we make the following observation: 

\begin{prop}
For every $n\geq 0$, there are natural maps of sheaves $$\Omega^1_X/B_{n,X} \rightarrow \mathrm{Br}_X[p].$$
\end{prop}
\begin{proof}
Note that $C(B_{n,X}) = B_{n-1, X} \subset B_{n,X}$. By Theorem \ref{thm p torsion br} we need to show that $(1-C)(B_{n,X}) \subset (1-C)(Z^1_X)$. In fact, it is enough to show that $(1-C)$ induces an isomorphism $B_n \xrightarrow{\sim} B_n$. This will be proved in the next lemma. 
\end{proof}
\begin{lemma}
For any $n \geq 0$ the natural map $(1-C) \colon B_{n,X} \rightarrow B_{n,X}$ is an isomorphism. 
\end{lemma}
\begin{proof}
First, note that this is well-defined since $C(B_n) = B_{n-1} \subset B_n$. We prove the statement by induction on $n$. When $n=1$, then $(1-C)$ is the identity on $B_{1,X}$ so there is nothing to prove. Now, the map $(1-C) \colon B_n \rightarrow B_n$ is injective because $B_n$ cannot contain logarithmic forms. So let $C_n$ be its cokernel. Then $(1-C)$ gives an injection of short exact sequences 
 \[\begin{tikzcd}
	0 & {B_{1,X}} & {B_{n+1,X}} & {B_{n,X}} & 0 \\
	0 & {B_{1,X}} & {B_{n+1,X}} & {B_{n,X}} & 0
	\arrow[from=1-1, to=1-2]
	\arrow[from=1-2, to=1-3]
	\arrow[hook, from=1-2, to=2-2]
	\arrow[from=1-3, to=1-4]
	\arrow[hook, from=1-3, to=2-3, "1-C"]
	\arrow[from=1-4, to=1-5]
	\arrow[hook, from=1-4, to=2-4]
	\arrow[from=2-1, to=2-2]
	\arrow[from=2-2, to=2-3]
	\arrow[from=2-3, to=2-4]
	\arrow[from=2-4, to=2-5]
\end{tikzcd}\]
By the snake lemma we get an exact sequence on cokernels, and we conclude by induction.
\end{proof}
We use this to define differential Brauer classes:
\begin{defi}
The differential Brauer group of level $n$ is
      $$ \Br^{\partial}_{n}(X) \coloneqq \im(H^0(X, \Omega^1_X/B_{n,X}) \rightarrow \Br(X)[p]).$$ For any section $\omega \in H^0(X, \Omega^1_X/B_{n,X})$ we denote by $[\omega] \in \Br^{\partial}_n(X)$ the induced element. 
\end{defi}
Note that for $m \leq n$ we have a diagram 
\begin{center}
    \begin{tikzcd}
    H^0(X, \Omega^1_X/B_{m,X}) \arrow[r] \arrow[d] & H^0(X, \Omega^1_X/B_{n,X}) \arrow[d] \\
    \Br(X)[p] \arrow[r, "="] & \Br(X)[p]
\end{tikzcd}
\end{center}
which shows that $\Br^{\partial}_{m}(X) \subset \Br^{\partial}_{n}(X) $. So we have an increasing filtration $$0  \subset \Br^{\partial}_{0} (X) \subset \Br^{\partial}_{1}(X) \subset  \cdots \subset \Br(X)[p]$$ and we define $\Br^{\partial}(X)$ to be the union of all the $\Br^\partial_n(X)$. 

Note that by the representability of flat cohomology with $\mu_p$-coefficients (see next section) we have, at least over an algebraically closed field $K$,  
$$\Br(X)[p] \cong (\text{finite $p$-torsion abelian group}) \oplus U(K)$$
where $U/K$ is a $p$-torsion unipotent algebraic group of finite type. We ask whether $\Br^\partial(X) = U(K) $ \footnote{This has now been proved in joint work with A. N. Skorobogatov and Y. Yang (in preparation).}. We shall show this for abelian and K3 surfaces.

\subsubsection{Differential Brauer group over closed fields.}
We study the filtration $\Br^{\partial}_n(X)$ over an algebraically
closed field by sheafifying the construction of differential Brauer classes. We begin with the following result.
\begin{prop}\label{prop rep brauer}
Let $K$ be an algebraically closed field of characteristic $p>0$,
and let $X/K$ be a smooth projective variety. On the big fppf site
$(\mathrm{Sch}/K)_{\mathrm{fppf}}$, let $\mathcal B$ be the sheaf
associated to the presheaf
\[
T\longmapsto\Br(X\times_K T)[p].
\]
Then $\mathcal B$ is represented by an affine commutative
$K$-group scheme of finite type, denoted by
$\widetilde{\Br}_{X/K}[p]$. Moreover, the natural map $\Br(X)[p]\longrightarrow\widetilde{\Br}_{X/K}[p](K)$
is an isomorphism.
\end{prop}

\begin{proof}
Let $ \pi \colon X\to\Spec(K)$ be the structure morphism. By \cite[Corollary~1.6]{arXiv:2107.11492}, the sheaf
$R^2 \pi _{*}\mu_p$ is represented by an affine
commutative $K$-group scheme $H$ of finite type. The Kummer
sequence gives
an exact sequence of sheaves
\[
\mathrm{Pic}_{X/K}\xrightarrow{p} \mathrm{Pic}_{X/K} \xrightarrow{\delta}H
\longrightarrow\mathcal B\longrightarrow0.
\]
We show that the fppf image of $\delta$ is a finite closed
subgroup scheme of $H$. Let $\mathrm{Pic}_{X/K}^0$ be the identity component of $\mathrm{Pic}_{X/K}$. By the homomorphism
theorem for group schemes of finite type over a field, the
scheme-theoretic image $\mathrm{Im}^0$ of $\mathrm{Pic}^0$ is a closed subgroup
scheme and $\mathrm{Pic}_{X/K}^0\to \mathrm{Im}^0$ is faithfully flat \cite[Theorem~5.39]{milne2017algebraicgroups}. Thus, $\mathrm{Im}^0$ also
represents its fppf image. Since $\mathrm{Im}^0$ is proper and $H$ is
affine, $\mathrm{Im}^0$ is finite. Choose now line bundles $L_1,\ldots,L_r$ on $X$ whose classes generate
the finitely generated group $\NS(X)$. Since $H$ is annihilated
by $p$, their images under $\delta$ generate a finite subgroup
of $H(K)$. Let $F\subseteq H$ be the corresponding finite
constant subgroup scheme. Every connected component of $\mathrm{Pic}_{X/K}$
is a translate of $\mathrm{Pic}_{X/K}^0$ by an integral linear combination of the $L_i$. Hence the fppf image of $\delta$ is $
\mathrm{Im}=\mathrm{Im}^0+F.$
This is a finite closed subgroup scheme of $H$, being a finite
union of translates of $\mathrm{Im}^0$. The quotient $H/ \mathrm{Im}$ is represented by an affine group
scheme of finite type; see
\cite[\href{https://stacks.math.columbia.edu/tag/03BM}
{Tag 03BM}]{stacks-project}. The exact sequence above identifies this quotient with
$\mathcal B$. Finally, every fppf covering of $\Spec(K)$ has a member
admitting a $K$-point since $K$ is algebraically closed. Consequently, sheafification does not
change the value of the presheaf at $\Spec(K)$, which proves
the last assertion about $K$-points.
\end{proof}

The group scheme $\widetilde{\Br}_{X/K}[p]$ need not be reduced.
We define
\[
\Br_{X/K}[p]
\coloneqq\bigl(\widetilde{\Br}_{X/K}[p]\bigr)_{\mathrm{red}}.
\]
Since $K$ is perfect, this is a closed subgroup scheme and is
smooth over $K$. Every morphism from a smooth $K$-scheme to
$\widetilde{\Br}_{X/K}[p]$ factors uniquely through its reduction.
Thus $\Br_{X/K}[p]$ represents the restriction of $\mathcal B$
to the smooth \'{e}tale site
$(\mathrm{Sm}/K)_{\mathrm{\acute{e}t}}$.
Here the restriction is taken after sheafification on the big
fppf site. In particular, $\Br_{X/K}[p](K)=\Br(X)[p].$ We write $\Br_{X/K}^0$ for the identity component
$\bigl(\Br_{X/K}[p]\bigr)^0$. Our aim now is to interpret the group morphism $H^0(X, \Omega^1_X/ B_{n,X}) \rightarrow \Br(X)[p]$ as a morphism of group schemes. For $n\geq0$, put $V_n \coloneqq H^0(X,\Omega^1_X/B_{n,X})$ and let $ b_n \coloneqq \dim_K V_n$. Let $\mathbf V_n
\coloneqq \Spec\bigl(\operatorname{Sym}_K(V_n^\vee)\bigr)$ be the associated vector group. Define $G_n \coloneqq \mathbf V_n^{(-n)}$.
The $n$-fold relative Frobenius gives a morphism $\epsilon_n:G_n\longrightarrow\mathbf V_n.$
It induces an isomorphism of additive groups
$G_n(K)\xrightarrow{\sim}V_n$.
For $n=0$, these definitions mean $G_0=\mathbf V_0$
and $\epsilon_0=\mathrm{id}$.
\begin{prop}\label{thm rep}
For every $n\geq0$, the functor on smooth $K$-schemes
\[
\widetilde{\mathcal F}_n(T)
=
H^0\left(X\times_K T,
\frac{\Omega^1_X\otimes_K\mathcal O_T^{p^n}}
{B_{n,X}\otimes_K\mathcal O_T^{p^n}}\right)
\]
is represented by $G_n$.
\end{prop}

\begin{proof}
Flat base change gives
\[
\widetilde{\mathcal F}_n(T)
=
V_n\otimes_K\Gamma(T,\mathcal O_T)^{p^n}.
\]
The relative Frobenius $\epsilon_n$ identifies $G_n(T)$
with this group: in coordinates, it sends $(a_i)$ to
$\sum_i\omega_i\otimes a_i^{p^n}$, where the $\omega_i$
form a basis of $V_n$.
This map is bijective because $T$ is reduced.
Since the identification is induced by $\epsilon_n$,
it is natural in $T$ and independent of the basis.
\end{proof}

\begin{lemma}
For every $n\geq0$, there is a canonical homomorphism
of $K$-group schemes
\[
\phi_n:G_n\longrightarrow\Br_{X/K}[p]
\]
such that $\phi_n(a)=[\epsilon_n(a)]$ for every $a\in G_n(K)$.
In particular,
\[
\phi_n(G_n(K))=\Br_n^\partial(X).
\]
\end{lemma}
\begin{proof}
Pullback gives natural maps
\[
\widetilde{\mathcal F}_n(T)
\longrightarrow
H^0(X\times_K T,
\Omega^1_{X\times_K T}/B_{n,X\times_K T})
\longrightarrow
\Br(X\times_K T)[p].
\]
The first map is well-defined because pullback preserves
$B_n$, and $B_n$ is stable under multiplication by
$p^n$-th powers. Passing to the sheafified Brauer functor and using
Proposition~\ref{thm rep}, we obtain a natural transformation
$G_n(-)\to\Br_{X/K}[p](-)$.
Yoneda's lemma gives the required homomorphism $\phi_n$.
On $K$-points, its construction gives
$\phi_n(a)=[\epsilon_n(a)]$.
Since $\epsilon_n$ is bijective on $K$-points, its image
is $\Br_n^\partial(X)$.
\end{proof}

In the following
statements, kernels, images and quotients are taken in the
category of fppf sheaves, or equivalently as the corresponding
group schemes when they are representable.

\begin{lemma}
For every $n\geq1$, there is a canonical homomorphism
\[
\psi_n:G_{n-1}^{(-1)}\longrightarrow G_n
\]
obtained from the inverse Cartier operator and the quotient $\pi_n:V_{n-1}\to V_n$.
\end{lemma}

\begin{proof}
The Cartier isomorphism gives a $p$-semilinear map
\[
C^{-1}:V_{n-1}
\xrightarrow{\sim}H^0(X,Z^1_X/B_{n,X})
\longrightarrow V_n.
\]
Consequently, $C^{-1}-\pi_n$ defines an additive polynomial
morphism of vector groups
\[
d_n:\mathbf V_{n-1}\longrightarrow\mathbf V_n.
\]
We define
\[
\psi_n=d_n^{(-n)}:
\mathbf V_{n-1}^{(-n)}
\longrightarrow\mathbf V_n^{(-n)},
\]
using $\mathbf V_{n-1}^{(-n)}=G_{n-1}^{(-1)}$ and
$\mathbf V_n^{(-n)}=G_n$.
\end{proof}

\begin{lemma}
For every $n\geq1$, the kernel of $\psi_n$ is a finite
group scheme.
\end{lemma}

\begin{proof}
Since $1-C$ is an automorphism of $B_{n,X}$,
Theorem~\ref{thm p torsion br} yields an exact sequence
on the Zariski site of $X$:
\begin{equation}\label{eq:logarithmic-mod-Bn}
0\longrightarrow\Omega^1_{\mathrm{log},X}
\longrightarrow Z^1_X/B_{n,X}
\xrightarrow{1-C}\Omega^1_X/B_{n,X}
\longrightarrow\Br_X[p]\longrightarrow0.
\end{equation}
The Cartier isomorphism identifies $\ker(d_n)(K)$ with
$H^0(X,\Omega^1_{\mathrm{log},X})$. The sequence
\[
0\longrightarrow\mathcal O_X^\times
\xrightarrow{(-)^p}\mathcal O_X^\times
\xrightarrow{\mathrm{dlog}}\Omega^1_{\mathrm{log},X}
\longrightarrow0
\]
and the equality $\Gamma(X,\mathcal O_X^\times)=K^\times$
give
\[
H^0(X,\Omega^1_{\mathrm{log},X})\cong\Pic(X)[p]
\]
which is finite. Thus $\ker(d_n)$ is an affine group scheme of finite type
with finitely many $K$-points, so it is finite. Its Frobenius
twist $\ker(\psi_n)$ is finite as well.
\end{proof}

\begin{lemma}
For every $n\geq1$, one has $\phi_n\circ\psi_n=0$.
Moreover, if $I_n$ denotes the image of $\psi_n$, then
\[
\ker(\phi_n)/I_n
\]
is a finite group scheme.
\end{lemma}

\begin{proof}
Let
\[
N_n=\ker\bigl(V_n\longrightarrow\Br(X)[p]\bigr).
\]
Naturality of relative Frobenius and the bijectivity of
relative Frobenius on $K$-points show that
\[
\epsilon_n\bigl(
\psi_n(G_{n-1}^{(-1)}(K))
\bigr)
=d_n(V_{n-1}).
\]
By \eqref{eq:logarithmic-mod-Bn}, this subgroup is contained
in $N_n$. Hence $\phi_n\circ\psi_n$ is zero on $K$-points.
Its source is reduced and of finite type over the algebraically
closed field $K$, so the composition is the zero morphism.
In particular, $I_n\subseteq\ker(\phi_n)$.

Factoring the middle map in \eqref{eq:logarithmic-mod-Bn}
through its image and taking Zariski cohomology gives an
injection
\[
N_n/d_n(V_{n-1})
\lhook\joinrel\longrightarrow
H^1(X,\Omega^1_{\mathrm{log},X}).
\]
Since $X$ is regular,
$H^2(X_{\mathrm{Zar}},\mathcal O_X^\times)=0$,
as follows from the flasque resolution by rational functions
and divisors. The logarithmic differential sequence therefore
gives
\[
H^1(X,\Omega^1_{\mathrm{log},X})
\cong\Pic(X)/p\Pic(X)
\cong\NS(X)/p\NS(X).
\]
The last isomorphism uses the divisibility of $\Pic^0_{X/K}(K)$.
In particular, this group is finite. The homomorphism theorem identifies $I_n$ with a closed
subgroup scheme of $G_n$. The quotient
$Q_n=\ker(\phi_n)/I_n$ is an affine group scheme of finite
type \cite[Proposition~5.29]{milne2017algebraicgroups}. Since $K$ is algebraically closed, quotient maps of
group schemes of finite type are surjective on $K$-points, hence, $Q_n(K)\cong N_n/d_n(V_{n-1})$
is finite. Thus $Q_n$ is finite.
\end{proof}

\begin{thm}\label{Main thm filtration}
Let $X/K$ be a smooth projective variety over an algebraically
closed field of characteristic $p>0$, and put
$b_n=h^0(X,\Omega^1_X/B_{n,X})$. Then:
\begin{enumerate}
\item For every $n\geq1$, there is a smooth connected
unipotent closed subgroup scheme
$\Gamma_n\subseteq\Br_{X/K}[p]$ such that
\[
\Gamma_n(K)=\Br_n^\partial(X),
\qquad
\dim\Gamma_n=b_n-b_{n-1}.
\]

\item The sequence $(b_n-b_{n-1})_{n\geq1}$ is nondecreasing,
and
\[
0\leq b_n-b_{n-1}\leq h^0(X,B^2_X).
\]
In particular, if $h^0(X,B^2_X)=0$, then
$\Br^\partial(X)=0$.
\end{enumerate}
\end{thm}

\begin{proof}
For every $n\geq0$, let $\Gamma_n$ be the image of $\phi_n$.
The homomorphism theorem shows that it is a closed subgroup
scheme and that $G_n\to\Gamma_n$ is faithfully flat.
Since $G_n$ is smooth and connected, its image is reduced
and connected, hence smooth over the perfect field $K$.
It is unipotent as a quotient of a vector group. Moreover,
$\Gamma_n(K)=\phi_n(G_n(K))=\Br_n^\partial(X)$.

For $n\geq1$, the finite kernel of $\psi_n$ gives
$\dim I_n=b_{n-1}$. Since $\ker(\phi_n)/I_n$ is finite,
\[
\dim\ker(\phi_n)=b_{n-1},
\qquad
\dim\Gamma_n=b_n-b_{n-1}.
\]
The inclusions
$\Br_n^\partial(X)\subseteq\Br_{n+1}^\partial(X)$
imply $\Gamma_n\subseteq\Gamma_{n+1}$, because these subgroup
schemes are reduced and $K$ is algebraically closed.
Their dimensions are therefore nondecreasing. Finally, the exact sequence
\[
0\longrightarrow Z^1_X/B_{n,X}
\longrightarrow\Omega^1_X/B_{n,X}
\xrightarrow{d}B^2_X\longrightarrow0
\]
and the Cartier isomorphism
\[
C:Z^1_X/B_{n,X}
\xrightarrow{\sim}\Omega^1_X/B_{n-1,X}
\]
give $0\leq b_n-b_{n-1}\leq h^0(X,B^2_X)$.
If $h^0(X,B^2_X)=0$, every $\Gamma_n$ with $n\geq1$ is
smooth, connected and zero-dimensional, hence trivial.
Since $\Br_0^\partial(X)\subseteq\Br_1^\partial(X)$,
this proves $\Br^\partial(X)=0$.
\end{proof}

\begin{rmk}
The implication
$h^0(X,B^2_X)=0\implies\Br^\partial(X)=0$
uses the hypothesis that $K$ is algebraically closed;
see Section~\ref{section curves}.
\end{rmk}

\begin{rmk}
For $n\geq1$, let
\[
H^0(X,B^2_X)_n
=\ker\bigl(
H^0(X,B^2_X)
\longrightarrow H^1(X,Z^1_X/B_{n,X})
\bigr).
\]
Then
\[
\dim\Gamma_n=\dim_K H^0(X,B^2_X)_n.
\]
\end{rmk}

\begin{cor}\label{Corollary b_n and Br}
Assume that
\[
b_n-b_{n-1}=\dim\Br_{X/K}[p]
\]
for some $n\geq1$. Then
\[
\Br_n^\partial(X)=\Br^\partial(X)
=\bigl(\Br_{X/K}[p]\bigr)^0(K).
\]
\end{cor}

\begin{proof}
Every $\Gamma_m$ is contained in the identity component
$\bigl(\Br_{X/K}[p]\bigr)^0$. Under the hypothesis,
$\Gamma_n$ has the same dimension as that component.
Both are smooth and connected, and $\Gamma_n$ is closed;
therefore they are equal. Taking $K$-points proves the result.
\end{proof}
\subsection{Examples} \label{varieties with diff sections}
In this section we give some examples of differential Brauer classes. The next proposition shows the different behavior of $\mathrm{Br}^{\partial}_0(X)$ depending on how pathological $H^0(X, \Omega^1_{X/K})$ is:
\begin{prop} \label{Prop Br_0}
Let $X$ be a smooth and projective variety over a field $K$. 
\begin{enumerate}
    \item Assume that all forms in $H^0(X, \Omega^1_{X/K})$ are closed (e.g., if the Hodge-to-de Rham spectral sequence degenerates). Then $\Br^\partial_{0}(X)$ consists of algebraic Brauer classes. 
    \item Assume that $K$ is algebraically closed and that $H^0(X, B^1_X) = H^0(X, Z^1_X) \subsetneq H^0(X, \Omega^1_X) $. Then, the kernel of the naturally induced map $$H^0(X, \Omega^1_X) / H^0(X, B^1_X) \rightarrow \Br(X)[p]$$
is isomorphic to the kernel of the Hodge cycle class map $c_1 \colon \NS(X)/ p \NS(X) \rightarrow H^1(X, Z^1_X)$ (and, in particular, it is finite). 
\end{enumerate}
\end{prop}
\begin{proof}
\begin{enumerate}
    \item In fact, over $\overline{K}$ the map $$(1-C) \colon H^0(X, Z^1_X) = H^0(X, \Omega^1_X) \rightarrow H^0(X, \Omega^1_X) $$
must be surjective since it is the difference of a linear isomorphism and a $p$-linear map between finite-dimensional $\overline{K}$-vector spaces.
\item 
Let $ \omega \in H^0(X, \Omega^1_X)$ be a global section and assume that $\omega$ is mapped to zero in $\Br(X)[p]$. The kernel of $H^0(X, \Omega^1_X) \rightarrow \Br(X)[p]$ is then $H^0(X, Z^1_X/ \Omega^1_{\log,X})$ and this last group fits into the exact sequence 
$$H^0(X, \Omega^1_{\log,X}) \rightarrow H^0(X, Z^1_X) \rightarrow H^0(X, Z^1_X / \Omega^1_{\log,X}) \rightarrow H^1(X, \Omega^1_{\log,X}) \rightarrow H^1(X, Z^1_X)  $$
where cohomology is taken with respect to the Zariski topology. 

But $H^1(X, \Omega^1_{\log,X}) = \NS(X) / p \NS(X)$ as the sequence defining $\Omega^1_{\mathrm{log},X}$ shows and the map $H^1(X, \Omega^1_{\log,X}) \rightarrow H^1(X, Z^1_X) $ factorises the cycle class map $c_1 \colon \NS(X)/ p \NS(X) \rightarrow H^2(X/K)$. Let $H \subset \NS(X)/ p$ be its kernel, which is a finite $p$-group. So, under the assumption, we get the sequence 
$$0 \rightarrow
H^0(X, B^1_X) / H^0(X, \Omega^1_{\log,X}) \rightarrow H^0(X, Z^1_X/ \Omega^1_{\log,X}) \rightarrow H \rightarrow 0;$$
now, note that $H^0(X, \Omega^1_{\log,X}) \subset H^0(X, Z^1_X) = H^0(X, B^1_X) $ and so $H^0(X, \Omega^1_{\log,X}) = 0 $ necessarily, because logarithmic forms cannot be exact. Hence the sequence above becomes 
$$0 \rightarrow
H^0(X, B^1_X) \rightarrow H^0(X, Z^1_X/ \Omega^1_{\log,X}) \rightarrow H \rightarrow 0$$
which proves the result.
\end{enumerate}
 
\end{proof}
The first assertion applies, for example, to curves and abelian varieties. The second assertion shows that $\mathrm{Br}^{\partial}_0(X)$ can also contain transcendental Brauer classes.

\subsubsection{Curves} \label{section curves} Let $C$ be a (smooth projective) curve over a perfect field $K$. We want to show that the map $H^0(C, \Omega^1_C) \rightarrow \Br(C)[p]$ is surjective, i.e., that $\Br_0^{\partial}(C) = \Br(C)[p]$. For simplicity, we assume that $C$ has a zero-cycle of degree one. We denote by $K^s$ a separable closure of $K$, and therefore $\bar{K} \cong K^s$ since $K$ is perfect. We denote by $C^s$ the base change of $C$ to $K^s$ and by $G_K$ the Galois group of $K^s/K$. Then, by our assumption on $C$ and Tsen's theorem, we have $$\Br(C) / \Br(K) =  H^1(G_K, \Pic(C^s)) = H^1(G_K, \Pic^0(C^s)).$$
Now we consider the short exact sequence 
$$0 \rightarrow \Pic^0(C^s)[p] \rightarrow \Pic^0(C^s) \xrightarrow{[p]} \Pic^0(C^s) \rightarrow 0$$ 
which yields the exact sequence $$0 \rightarrow \Pic^0(C)/ p \Pic^0(C) \rightarrow H^1(G_K, \Pic^0(C^s)[p]) \rightarrow \Br(C)[p] \rightarrow 0.$$
This describes the Brauer group of $C$; note, in particular, that $\Br(C)[p] = 0$ whenever $\Pic^0(C)$ is supersingular. 

We now look at $\Br_0^\partial(C)$. Note that $\Omega^1_C = \Omega^1_{C/K} = Z^1_{C/K} = Z^1_{C}$ where the middle identification comes from $\dim(C) = 1$ and the others from the fact that $K$ is perfect. Also, note that the Cartier operator identifies $H^0(C, \Omega^1_C/ B_{n,C}) \cong H^0(C, \Omega^1_C)$ for every $n \geq 0$, so it is enough to determine $H^0(C, \Omega^1_C)$ (but we shall not analyse the induced action of the Cartier operator on the Brauer group). The short exact sequence on $C_{\et}$: 
$$0 \rightarrow \Omega^1_{\log, C} \rightarrow \Omega^1_C \xrightarrow{1-C} \Omega^1_C \rightarrow 0$$
yields the sequence 
$$0 \rightarrow \frac{H^0(C, \Omega^1_C)}{(1-C)H^0(C, \Omega^1_C)} \rightarrow H^2_{\fppf}(C, \mu_p) \rightarrow H^1(C, \Omega^1_C) \xrightarrow{1-C} H^1(C, \Omega^1_C).$$
Now, $H^1(C, \Omega^1_C) \cong K$ is generated by the cycle class of a point $[pt] \in H^1(C, \Omega^1_C)$ and $C ( [pt] ) = [pt]$. From this it follows that $\ker(1-C) = \Ff_p \cdot [pt]$. Hence, we can write the sequence as 
$$0 \rightarrow \frac{H^0(C, \Omega^1_C)}{(1-C)H^0(C, \Omega^1_C)} \rightarrow H^2_{\fppf}(C, \mu_p) \rightarrow \Ff_p \cdot [pt] \rightarrow 0.$$
Base-changing to $K^s$ yields the commutative diagram 
\[\begin{tikzcd}
	{H^2_{\fppf}(C, \mu_p) } & { \Ff_p \cdot [pt]} \\
	{H^2_{\fppf}(C^s, \mu_p) } & { \Ff_p \cdot [pt]}
	\arrow[from=1-1, to=1-2]
	\arrow[from=1-1, to=2-1]
	\arrow["{=}", from=1-2, to=2-2]
	\arrow[from=2-1, to=2-2]
\end{tikzcd}\]
and via Kummer theory we get $H^2_{\fppf}(C^s, \mu_p)  = \Ff_p \cdot [pt]$ so that the bottom map is an isomorphism. This then shows that the kernel of $H^2_{\fppf}(C, \mu_p) \rightarrow \Ff_p \cdot [pt]$ consists precisely of all the cohomological classes that die over $K^s$. So if $(\Pic(C)/ p \Pic(C))^0 \subset \Pic(C)/ p \Pic(C)$ is the kernel of the degree map, we have $\Pic(C)/ p \Pic(C) \rightarrow \Ff_p$ we have a short exact sequence 
$$0 \rightarrow (\Pic(C)/ p \Pic(C))^0 \rightarrow  \frac{H^0(C, \Omega^1_C)}{(1-C)H^0(C, \Omega^1_C)} \rightarrow \Br(C)[p] \rightarrow 0$$
which shows in particular that $\Br(C)[p] = \Br^{\partial}_0(C)$.
\subsubsection{Abelian and K3 surfaces}
We shall use Corollary \ref{Corollary b_n and Br} to determine the differential Brauer group of supersingular abelian surfaces and supersingular K3 surfaces over an algebraically closed field $K$. More precisely, let $\Br_{X/K}[p]$ be the group scheme defined in Proposition \ref{prop rep brauer} and the discussion afterwards. Our aim is to show that for supersingular abelian or K3 surfaces the map $G_n \rightarrow \Br_{X/K}[p]$ is surjective onto the connected component of the identity, for sufficiently large $n$. Note that for both abelian and K3 surfaces the connected component of the identity of $\Br_{X/K}[p]$ is $\mathbb{G}_a$. For K3 surfaces this was originally discovered by Artin in the seminal paper \cite{MR0371899} (see also \cite{bragg2019twistorspacessupersingulark3}[Proposition 2.2.4]). In fact, Artin showed more, namely that $\tilde{\Br}_{X/K}$ itself is smooth with connected component $\mathbb{G}_a$.

Let $A/K$ be a supersingular abelian surface. By Proposition \ref{Prop Br_0} we know that $H^0(A, \Omega^1_A) \rightarrow \Br(A)$ is zero. So if $\Br^{\partial}_n(A) \neq 0$ then necessarily $n \geq 1$. We use the results in \cite{MR1827020} to compute the various $b_n \coloneqq h^0(A, \Omega^1_A/ B_{n,A})$. 
We have only two cases to consider:
\begin{enumerate}
    \item $A$ is superspecial, so $A = E \times E$ for $E$ a supersingular elliptic curve and $A$ has Artin invariant $1$. In this case we have $b_n = 2 + n$ for every $n \geq 0$ and hence $b_{n} - b_{n-1} = 1$. This shows that $\Br^\partial_1(A) = \Br_{A/K}^0(K)$ due to Corollary \ref{Corollary b_n and Br}.
    \item The Artin invariant is $2$ and $A$ is only isogenous to $E \times E$ but not isomorphic to it. In this case, we have $b_0 = 2$, $b_1 = 2$, $b_n = 1+ n$ for $n \geq 2$. Hence, $b_1 - b_0 = 0$ and so $\Br^\partial_1(A) = 0$ but $b_2 - b_1 = 1$ and so $\Br^\partial_2(A) = \Br_{A/K}^0(K)$ by the same corollary. 
\end{enumerate}
Let now $X/K$ be a supersingular K3 surface. Let $a>0$ be its Artin invariant. Then \cite[Theorem~6.2 and Corollary~11.12]{MR1776939} says that $\dim_{K} H^1(X, B_{n,X}) = n$ and that the map $ H^1(X, B_{n,X}) \rightarrow H^1(X, \Omega^1_X)$ is injective for $n < a$ and has a one-dimensional kernel for $n = a$. Thus $b_n = 0$ for $n<a$ and $b_n = n+1-a$ for $n \geq a$. Thus $b_n - b_{n-1} = 1$ for $n = a$ and so $\Br^{\partial}_a(X) = \Br_{X/K}^0(K)$.
\subsubsection{Pairing infinitesimal torsors} \label{sub examples}
Another useful way to obtain differential Brauer classes is by pairing torsors under infinitesimal group schemes. For $X/K$ smooth one has $H^1(X_{\mathrm{fppf}}, \alpha_p) \cong H^0(X, B^1_X)$ again by applying \cite{MR559531}[Theorem 3.9] as in the beginning of Section \ref{section p torsion}. One constructs a pairing
$$H^0(X, B^1_X) \times H^0(X, B^1_X) \rightarrow H^0(X, \Omega^1_X/ B^1_X) $$
as follows: pick two sections $\omega, \sigma \in H^0(X, B^1_X)$. For an affine open cover $U_i = \Spec(A_i)$ we let $\omega$ and $\sigma$ be represented respectively by local functions $f_i,g_i \in A_i$ such that $d(f_i) - d(f_j) = 0 = d(g_i) - d(g_j) $ in $\Omega^1_{A_{ij}}$ where $\Spec(A_{ij}) = U_{ij}$. Since $X/K$ is smooth, this is equivalent to the fact that $f_i - f_j \in A_{ij}^p$ (and similarly for $g_i$).  Now, the various $g_i d(f_i) \in \Omega^1_{A_i}$ glue to a section of $H^0(X, \Omega^1_X/ B^1_X)$. In fact, on the overlaps $U_{ij} = \Spec(A_{ij})$ we can write $f_i - f_j = s_{ij}^p$ and $g_i - g_j = t_{ij}^p$, and hence $$g_i d(f_i) - g_j d(f_j) = t_{ij}^p d(f_i) +  g_j d(f_i) - g_j d(f_j) = d(t_{ij}^p f_i).$$
We thus define  $\langle \omega, \sigma \rangle$ as the section of $\Omega^1/B_1$ represented by $\{ g_i d f_i \}_i$. This is clearly bilinear, and moreover $$\langle \omega, \sigma \rangle  + \langle \sigma , \omega \rangle = \{ g_i d f_i + f_i d g_i \} = \{ d(f_i g_i) \}.$$ Thus the pairing is skew-symmetric. If $p\neq 2$, it is
also alternating. In characteristic two, alternation can fail:
on $X=\mathbb A^1_K$, the class $\langle dx,dx\rangle=x\,dx\pmod{B_{1,X}}$ is nonzero because $C(x\,dx)=dx$.
Note that one also has a perfect pairing of group schemes $\alpha_p \times \alpha_p \rightarrow \mu_p$ which induces a pairing $$H^1(X, \alpha_p) \times H^1(X, \alpha_p) \rightarrow \Br(X)[p].$$
\begin{prop}
    We have a commutative diagram of pairings
  \begin{center}
       \begin{tikzcd}
 H^0(X, B^1_X) \times H^0(X, B^1_X) \arrow[r] \arrow[d]
    & H^0(X, \Omega^1_X/ B^1_X) \arrow[d] \\
  H^1(X, \alpha_p) \times H^1(X, \alpha_p) \arrow[r]
& \Br(X)[p]. \end{tikzcd}
  \end{center} 
\end{prop}
The proof is a direct computation and can be found, for example, in \cite{beli2018analoguespnthhilbertsymbol} (it is enough to show commutativity at the level of function fields). This construction generalizes to every $n\geq1$ using the differentials $D_n\colon W_n\to B_n$. The group scheme associated with $B_n$ is $\ker(F:W_n\to W_n)$. We can interpret a section of $H^0(X,B_n)$ on an affine open cover $U_i = \Spec(A_i)$ as given by $\underline{f}_i \in W_n(A_i)$ such that $\underline{f}_i - \underline{f}_j \in FW_n(A_{ij})$. Then, one constructs the pairing 
$$H^0(X,  B_{n,X}) \times H^0(X, B_{n,X}) \rightarrow H^0(X, \Omega^1_X/ B_{n,X}) $$
using the analogue of the formula above $\langle \omega, \sigma \rangle = \{ D_n(\underline{f}_i)g_{0,i}^{p^{n-1}} \} \in H^0(X, \Omega^1_X/B_n)$, where $\sigma$ is represented by $\underline{g}_i = (g_{0,i}, g_{1,i}, \cdots, g_{{n-1},i})$. This pairing is again well defined and skew-symmetric,
and it is alternating when $p\neq 2$. For example, we have
\begin{prop}
    Let $A$ be a supersingular abelian surface over a closed field of odd characteristic. Then 
    \begin{enumerate}
        \item If $A$ is superspecial, 
        $$H^0(A, B_{1,A}) \times H^0(A, B_{1,A})  \rightarrow \Br(A) $$
        is surjective;
        \item If $A$ is not superspecial, 
        $$H^0(A, B_{2,A}) \times H^0(A, B_{2,A})  \rightarrow \Br(A) $$
        is zero.
    \end{enumerate}
\end{prop}
\begin{proof}
\begin{itemize}
    \item Here $H^0(A,B_{1,A})=H^0(A,\Omega^1_A)$ is two-dimensional.
Choose a basis $\omega_1,\omega_2$, and write $\omega_i=df_i$
generically. The class $\beta=\langle\omega_1,\omega_2\rangle$ is
represented by $f_2df_1$, so $d\beta=df_2\wedge df_1\neq0$.
Since $b_1-b_0=1$, it spans the quotient of
$H^0(A,\Omega^1_A/B_{1,A})$ by $H^0(A,Z^1_A/B_{1,A})$.
The latter is the image of $C^{-1}-1$, because
$H^0(A,\Omega^1_A)=H^0(A,B_{1,A})$, and therefore maps to zero
in $\Br(A)$. Thus every class of $\Br^\partial_1(A)$ is of the form
$[\lambda\beta]$, with $\lambda\in K$, and is realized by
$\langle\lambda\omega_1,\omega_2\rangle$.
The earlier computation gives
$\Br^\partial_1(A)=\Br_{A/K}^0(K)$.
Kummer comparison with a supersingular K3 surface
\cite[Proposition~4.1]{skorobogatov2024boundednesspprimarytorsionbrauer}
and Artin's description \cite{MR0371899} identify this with the whole
Brauer group of $A$.
    \item In this case, $H^0(A, B_{1,A})$ is one dimensional and spanned by some $\omega_1 = df_1$ whereas $H^0(A, B_{2,A})$ is two dimensional and spanned by $\omega_1$ and some $\omega_2 = d g_1 + g_0^{p-1} dg_0$. Since the pairing is alternating and $p$ is odd we have $\langle \omega_i , \omega_i \rangle = 0$; moreover, by the definition of the pairing, we also have $\langle \omega_1 , \omega_2 \rangle = 0$ (since $\omega_1$ has no $f_0^{p-1} d f_0$ summand). 
\end{itemize} 
\end{proof}
\begin{rmk}
The fact that $H^1(A, \alpha_p) \otimes H^1(A, \alpha_p) \rightarrow \Br(A)[p]$ is surjective when $A$ is superspecial is also proved in \cite{skorobogatov2024boundednesspprimarytorsionbrauer} via different methods. 
\end{rmk}
\subsection{Descending Brauer classes over non-closed fields} \label{section abs vs rel}
To produce differential Brauer classes over non-closed fields we need to study global sections of $\Omega^1_{X}/ B_{n,X}$. In general, we understand sections of $\Omega^1_{\overline{X}}/ B_{n,\overline{X}}$ better, and we would like to know when these descend to sections of $\Omega^1_X / B_{n,X}$. For $n=0$ the answer lies in the Kodaira-Spencer map: since $X \rightarrow \Spec(K)$ is smooth we have a short exact sequence
\begin{equation} \label{sequence omega}
    0 \rightarrow \Omega^1_K \otimes \Oo_X \rightarrow \Omega^1_X \rightarrow \Omega^1_{X/K} \rightarrow 0.
\end{equation}
The first connecting morphism (which can be interpreted as pairing with the Kodaira-Spencer class of \( X/K \))
\[
\partial \colon H^0(X, \Omega^1_{X/K}) \rightarrow H^1(X, \mathcal{O}_X) \otimes \Omega^1_K
\]
can in general be nontrivial, and it may happen, for instance, that $H^0(X, \Omega^1_{X/K}) \neq 0$ but $H^0(X, \Omega^1_X) = \Omega^1_K$, or that $\Br^\partial_0(X) \subset \Br(K)$ but $\Br^\partial_0(\overline{X}) \neq 0$. We give an elementary example of this; see also \cite{zbMATH01584207}[Section 5.1].
\begin{exs}
Take \( K = \mathbb{F}(t) \), where \( \mathbb{F} \) is a perfect field, and let \( E_t/K \) be the Legendre family \( y^2 = x(x-1)(x-t) \) over \( K \), assuming \( p \geq 5 \). The differential form \( \omega = dx/(2y) \) defines a global section of \( \Omega^1_{E_t/K} \). Since we could not find a reference we show that \( \partial(\omega) \neq 0 \) for the reader’s convenience. Let \( P(x,t) = x(x-1)(x-t) \), and denote by \( P_x \) and \( P_t \) its partial derivatives. At the point at infinity $[0 \colon 1 \colon 0]$ introduce coordinates \( u, v \) with \( vx = u \) and \( vy = 1 \), so that the equation for $E_t$ becomes \( v = Q(u,v,t) := u(u-v)(u - tv) \). Consider the Zariski opens \( U_1 = E_t \setminus \{(x,y) \colon x = 0,1,t\} \) and \( U_2 = E_t \setminus (\{ (x,y) \colon P_x(x,t) = 0\} \cup \{ \infty \}) \), which cover \( E_t \) (in particular, $\infty \in U_1$). Let \( U_{12} = U_1 \cap U_2 \). We claim that \( \omega = dx/(2y) \) extends regularly over \( U_1 \). So we only have to check that it extends at infinity. This is the usual computation taking into account that $t$ is not constant. In the coordinates \( u,v \), we compute:
\( 2v \cdot \omega = v\, du - u\, dv, \)
so we must show that \( u\, dv \) is divisible by \( v \). Using the equation \( v=Q(u,v,t) \), we get:
\( u\, dv = \frac{1}{1 - Q_v}(u Q_u\, du + u Q_t\, dt), \)
and \( Q_t = -uv(u - v) \) is divisible by \( v \). It remains to check that \( u Q_u\, du \) is divisible by \( v \). Writing \( u Q_u = 3u^3 + v F(u,v,t) \) for a polynomial $F$ and noting that \( v = Q = u^3 + v G(u,v,t) \) for some other polynomial $G$ we find:
\( \frac{u^3}{v} = 1 - G(u,v,t) \quad \Rightarrow \quad u Q_u\, du = v(3 - 3G + F)\, du, \)
which shows that \( \omega \) is regular on \( U_1 \). Now, as elements of $\Omega^1_{K(E_t)/K}$, we have an equality $\omega = dy / P_x$ over $U_{12}$. The same expression $dy / P_x$ extends then to a global form of $\Omega^1_{E_t}$ over $U_2$. Over \( U_{12} \), viewed as a section of \( \Omega^1_{E_t}(U_{12}) \), we compute the difference
\( dx/(2y) - dy / P_x = -\frac{P_t}{2y P_x} dt \)
which is a Čech cocycle representing \( \partial(\omega) \). To see that this is nontrivial, note that over $U_2$ the rational function $-P_t/(2yP_x)$ has only a simple pole at $x = t$ (and simple zeros at $x = 0,1$). If \( -\frac{P_t}{2y P_x} = f_1 - f_2 \) for \( f_i \in \mathcal{O}(U_i) \), then \( f_1 \) would have exactly one simple pole in \( U_2 \), implying that \( f_1 \) has degree one, which is impossible. This shows that $\Br^{\partial}_0(E_t) \subset \Br(K)$. On the other hand, we have $\Br^{\partial}_0((E_t)_{K^{\mathrm{perf}}}) = \Br((E_t)_{K^{\mathrm{perf}}})[p]$ where $K^{\mathrm{perf}}$ is the perfection of $K$; this follows from Section \ref{section curves}. 
\end{exs}
On the other hand, it is a classical observation that the sequence \eqref{sequence omega} splits whenever $X$ descends to $K^p$:
\begin{defi} \label{defi descent}
We say that $X/K$ descends along the Frobenius of $K$ if there is $\tilde{X}/K$ with $\tilde{X}^{(1)} \coloneqq \tilde{X} \times_{F_K} \Spec(K) \cong X$. In this case, we denote by $F_X \colon \tilde{X} \rightarrow X$ the relative Frobenius of $\tilde{X}$ and, for any $K$-algebra $R$, we put $X(R)^p \coloneqq F_X(\tilde{X}(R))$. 
\end{defi}
\begin{prop} \label{prop split diff}
If $X/K$ descends along the Frobenius of $K$, then the exact sequence \eqref{sequence omega} splits. 
\end{prop}
\begin{proof}
Denote by $F_K:X\to\tilde X$ the projection induced by the Frobenius
of $K$. Locally, write $f=\sum_i a_iF_K^*(g_i)$ with $a_i\in K$
and $g_i\in\mathcal O_{\tilde X}$, and set
$d_K(f)=\sum_iF_K^*(g_i)\,da_i$.
The tensor-product relations respect this formula because $d(c^p)=0$
for every $c\in K$. It therefore defines a derivation with values in
$\Omega^1_K\otimes_K\mathcal O_X$.
The induced map $\Omega^1_X\to\Omega^1_K\otimes_K\mathcal O_X$
is a retraction of the inclusion in \eqref{sequence omega}, giving
the required splitting.
\end{proof}
In particular, we have the following corollary. 
\begin{cor}
Up to a finite (inseparable) field extension of $K$ of exponent one, we can assume that \eqref{sequence omega} splits. 
\end{cor}
We now look at the sheaves $\Omega^1_X / B_{n,X}$. 
\begin{defi}
Let $\pi:\Omega^1_X\to\Omega^1_{X/K}$ be the natural map.
We define $B_{n,X/K}\coloneqq K\cdot\pi(B_{n,X})$, a subsheaf of
$\Omega^1_{X/K}$.
\end{defi}
These relative sheaves commute with field extension, as follows from
the relative Cartier isomorphism. In particular,
$B_{n,X/K}\otimes_K\bar K\cong B_{n,\bar X}$.
Consequently, $H^0(X,\Omega^1_{X/K}/B_{n,X/K})$ is nonzero
whenever its geometric counterpart is nonzero.
The next proposition lifts relative classes to absolute classes
after Frobenius descent.
\begin{prop}\label{prop morphism descent}
Let $n\geq1$ and suppose that $X\cong\tilde X^{(n)}$ for a smooth
$K$-variety $\tilde X$. Pullback along the projection induced by $F_K^n$
gives a homomorphism $(F_K^n)^*$ from
$H^0(\tilde X,\Omega^1_{\tilde X/K}/B_{n,\tilde X/K})$ to
$H^0(X,\Omega^1_X/B_{n,X})$.
Its composition with the natural map to
$H^0(X,\Omega^1_{X/K}/B_{n,X/K})$ is injective, and is bijective
when $K$ is perfect.
\end{prop}
\begin{proof}
Let $\omega\in H^0(\tilde X,\Omega^1_{\tilde X/K}/B_{n,\tilde X/K})$.
On an affine open cover $U_i$ of $\tilde X$, lift $\omega$ to absolute
forms $\tilde\omega_i$. On overlaps, their differences belong to
$K\cdot B_{n,U_{ij}}+\Omega^1_K\otimes_K\mathcal O_{U_{ij}}$.
Pullback along $F_K^n$ kills the second summand.
The first is generated by terms $t f^{p^i-1}df$ with $t\in K$
and $0\leq i<n$. Writing $g=(F_K^n)^*f$, their pullbacks are
$t^{p^n}g^{p^i-1}dg=(t^{p^{n-i}}g)^{p^i-1}d(t^{p^{n-i}}g)$,
which belong to $B_n$. Thus the pulled-back forms glue modulo $B_n$.
After projection to relative differentials, the map is induced by flat
base change along $F_K^n$, so it is injective on global sections.
When $K$ is perfect, this base change is an isomorphism.
\end{proof}
\begin{cor}\label{cor splitting}
Let $n\geq1$ and assume that $X$ descends to $K^{p^n}$.
The image of $H^0(X,\Omega^1_X/B_{n,X})$ in
$H^0(\bar X,\Omega^1_{\bar X}/B_{n,\bar X})$ spans the latter
as a $\bar K$-vector space.
\end{cor}
\begin{proof}
Apply Proposition~\ref{prop morphism descent} and extend scalars to
$\bar K$. Its relative composition becomes an isomorphism, so the
images of the relative sections on $\tilde X$ span the geometric space.
\end{proof}
\section{The Brauer--Manin obstruction} \label{section BM}
In this second and final part of the article we prove Theorem \ref{theorem 1 intro} and its consequences. Let $K$ be a global function field with constant field $\Ff \subset K$ and let $\mathbb{A}_K \coloneqq {\prod_{v}}' K_v$ be the adeles of $K$. We put
\[
\Omega^1_{\mathbb{A}_K} \coloneqq {{\prod}_{v}}' \Omega^1_{K_v},
\]
where the restricted product is taken with respect to $\Omega^1_{\Oo_{v}} \subset \Omega^1_{K_v}$. We have a Cartier operator
\begin{equation} \label{eq: cartier adeles}
    C = \prod C_{K_v} \colon \Omega^1_{\mathbb{A}_K} \rightarrow \Omega^1_{\mathbb{A}_K}
\end{equation}
and we define $B_{n,\mathbb{A}_K}$ inductively using $C$ as in Section \ref{first section}. One has
\[
B_{n,\A_K} = {{\prod}_{v}}'B_{n,K_{v}},
\]
where again the restricted product is taken over $B_{n,\Oo_{v}}$. Note that any $ \omega \in B_{n, K_v}$ has zero residue. Now, for any integral ring $A$ with a $p$-basis and fraction field $F$, one easily checks that $\Omega^1_{A} \cap B_{n,F} = B_{n,A}$. Hence, one has an injection $\Omega^1_A/ B_{n,A} \subset \Omega^1_F/ B_{n,F}$.

\begin{lemma}
One has 
\[
\Omega^1_{\A_K}/ B_{n, \A_K} \cong {\prod}'_v  \Omega^1_{K_v}/ B_{n,K_v},
\]
where the product is restricted over $\Omega^1_{\Oo_{v}}/ B_{n,\Oo_{v}} \subset \Omega^1_{K_v}/ B_{n,K_v}$.
\end{lemma}

\begin{proof}
The natural map $\Omega^1_{\A_K} \rightarrow {\prod}'_v  \Omega^1_{K_v}/ B_{n,K_v}$ is clearly surjective. If an element $(\omega_v)$ belongs to the kernel, then for every $v$ we must have $\omega_v \in B_{n,K_v}$. So for almost all $v$ we have $\omega_v \in \Omega^1_{\Oo_{v}} \cap B_{n,K_v} = B_{n,\Oo_{v}},$ hence $(\omega_v) \in B_{n,\mathbb{A}_K}$.
\end{proof}

By a small abuse of notation, if $x_v \colon \Spec(K_v) \rightarrow X$ and $\omega \in H^0(X,\Omega^1_X/ B_n)$ we denote by $\omega|_{x_v} \in \Omega^1_{K_v}$ the pullback of $\omega$ via $x_v$. 

\begin{lemma}
For any adelic point $(x_v) \in X(\A_K)$ and any section $\omega \in H^0(X, \Omega^1_X/B_n)$ we have $(\omega_{| x_v}) \in \Omega^1_{\A_K}/B_{n,\mathbb{A}_K}$.    
\end{lemma}

\begin{proof}
This is a standard spread-out argument. Choose a smooth affine model $U/ \Ff_{q}$ of the function field $K$ such that $X/K$ spreads out to a smooth projective morphism $\mathcal{X} \rightarrow U$. Up to shrinking $U$, we can also assume that $\omega$ extends to a section $\tilde{\omega} \in H^0(\mathcal{X}, \Omega^1_{\mathcal{X}} /B_{n,\mathcal{X}})$. Now, for any place $v$ of $K$ in $U$, the component $x_v \colon \Spec(K_v) \rightarrow X$ extends by completeness to a map $\tilde{a}_v \colon \Spec(\Oo_{v}) \rightarrow \mathcal{X}$, from which it follows that $\omega_{| x_v} \in \Omega^1_{\Oo_{v}}/ B_{n,\Oo_{v}}$. 
\end{proof}

\begin{proof}[Proof of Theorem \ref{theorem 1 intro}]
We have $(1-C)(B_{n, \mathbb{A}_K}) = B_{n, \mathbb{A}_K} $ and moreover
\[
\coker \big( (1-C) \colon \Omega^1_{\A_K} \rightarrow \Omega^1_{\A_K} \big) = \bigoplus_v \Omega^1_{K_v} / (1-C) \Omega^1_{K_v},
\]
where the fact that it lands in the direct sum is due to $(1-C) \Omega^1_{\Oo_{v}} = \Omega^1_{\Oo_{v}}$ which follows from the explicit computation of Section \ref{section: K_q((t))}. Hence, we have natural maps 
\[
\Omega^1_{\mathbb{A}_K} / B_{n, \mathbb{A}_K} = {\prod}'_{v} \Omega^1_{K_v} / B_{n, K_v} \rightarrow \bigoplus_v \Omega^1_{K_v} / (1-C) \Omega^1_{K_v} = \bigoplus_v \Br(K_v)[p],
\]
In fact, we also showed that the map
\[
\Omega^1_{\mathbb{A}_K} \rightarrow \bigoplus_v \Br(K_v)[p] \cong \bigoplus_v \Ff_p
\]
is given by
\[
(\omega_v) \mapsto \bigoplus_v \Tr_{\Ff(v)/\Ff_p}(\Res (\omega_v))
\]
where $\Ff(v)$ is the residue field of $v$. Now, the fact that a point $(x_v)$ is unobstructed for all the classes of the form $[t^{p^n} \omega]$ means that for every $t \in K$ we have 
\[
\sum_{v} \Tr_{\Ff(v)/\Ff_p}(\Res (t^{p^n} \omega_{| x_v})) = 0 .
\]
But $\Res(\omega) = \Res(C(\omega))^{1/p}$ and hence 
\[
\sum_{v}  \Tr( \Res (t^{p^n} \omega_{| x_v})) = \sum_{v}  \Tr(\Res (C^n(t^{p^n} \omega_{| x_v}))) = \sum_{v}  \Tr 
(\Res (tC^n(\omega_{| x_v}))).
\]
Finally, the proof of \cite{MR0670072}[Theorem 2.1.1] shows that $\omega \in \Omega^1_{\mathbb{A}_K}$ satisfies 
\[
\sum_{v} \Tr 
(\Res (t \omega)) = 0 \text{  for all  } t\in K
\]
if and only if $\omega \in \Omega^1_K$, which shows that $C^n((\omega_{| x_v}))$ is global.  
\end{proof}
Compare this with the computations in \cite{MR3062870}[Proposition 2.1], which inspired this argument. 

\subsection{Relations with \texorpdfstring{$F$}{F}-descents} \label{F-descent sub}
Voloch pointed out to us a connection between the result above and the results proved in \cite{MR4448590}, where the authors reinterpret the Frobenius descent obstruction of constant curves in terms of differential forms. Their computations also draw on the results of \cite{arXiv:1904.03687} and \cite{MR419450}. Let $A/\Ff$ be an abelian variety over a finite field and consider the relative Frobenius $F \colon A^{(-1)} \rightarrow A$, which is a $\ker(F)$-torsor over $A$. Now, fix a global field $\Ff \subset K$ (with $\Ff$ algebraically closed in $K$). Any local point $x_v \in A(K_v)$ yields by functoriality an element $F^{-1}\{x_v\} \in H^1_{\mathrm{fppf}}(K_v, \ker(F))$, and one says that the adelic point $(x_v) \in A(\A_K)$ survives the $F$-descent obstruction if $$(F^{-1}(x_v)) \in \mathrm{im}(H^1_{\mathrm{fppf}}(K, \mathrm{ker}(F)) \rightarrow \prod_v H^1_{\mathrm{fppf}}(K_v, \ker(F))).$$ 
\begin{prop}
In the above situation we have 

$$A(\mathbb{A}_K)^{\Br_0^{\partial}(A)} = A(\mathbb{A}_K)^{F-\mathrm{desc}} $$
\end{prop}
\begin{proof}
The proof of \cite{MR4448590}[Lemma 3.2] shows that 
$$A(\A_K)^{F-\mathrm{desc}} = \{(x_v) \in A(\A_K) \colon (\omega_{| x_v})_v \in \Omega^1_K \text{ for every } \omega \in H^0(A, \Omega^1_A) \}.$$
\end{proof}
Let now $X/ \mathbb{F}$ be a smooth projective variety and let $A$ be its Albanese variety. Then $X$ has a zero-cycle of degree one, and so we can choose a morphism $\mathrm{alb} \colon X \rightarrow A$. We can thus talk about the $F$-descent obstruction on $X$, simply by saying that $(x_v) \in X(\mathbb{A}_K)$ survives the $F$-descent obstruction if its image in $A(\mathbb{A}_K)$ does. 
\begin{cor}
There is a natural inclusion 
$$X(\mathbb{A}_K)^{\Br_0^{\partial}(X)} \subseteq X(\mathbb{A}_K)^{F-\mathrm{desc}}$$
with equality holding if the natural map $H^0(A, \Omega^1_A) \rightarrow H^0(X, \Omega^1_X)$ is an isomorphism. 
\end{cor}
In positive characteristic, the map $H^0(A, \Omega^1_A) \rightarrow H^0(X, \Omega^1_X)$ is always injective by a result of Igusa \cite{zbMATH03114814}, but it may fail to be surjective. In fact, Igusa gave the first example of a smooth surface $X$ such that $H^0(X, \Omega^1_X)$ is two dimensional but its Albanese variety is an elliptic curve in \cite{zbMATH03114813}. This says that the two obstructions coincide in the former situation but may differ for ``pathological'' varieties

\subsection{The evaluation map} \label{sec evaluation map}
We fix a global field $K$ of positive characteristic. We use the description in Theorem \ref{theorem 1 intro} to study the evaluation map coming from classes of $\Br^\partial_n(X)$. The main theme is that if $\alpha \in \Br(X)$ comes from some section $\omega \in H^0(X, \Omega^1_X/ B_{n,X})$ then all the evaluation maps associated with the classes $[t^{p^n}\omega]$ can be assembled into a single analytic object. We begin by explaining the case $n=0$. We recall that if $U \subset K_v^n$ is a $v$-adic open subset, a continuous function $ f \colon U \rightarrow K_v$ is analytic if it can be expressed around every point $u \in U$ as a convergent power series with coefficients in $K_v$. Similarly, we say that a function $f \colon U \rightarrow K_v^m$ is analytic if all its coordinates are analytic. Note that analytic maps need not be open. Also, note that the map $d \colon K_v \rightarrow \Omega^1_{K_v} \cong K_v$ is not analytic. But we can force it to be analytic as follows: consider both $K_v$ and $\Omega^1_v$ as $K_v^p$-vector spaces in a natural way. Then both have dimension $p$, and the map $d$ becomes analytic (because it is linear). 
\begin{defi}
Let $X_n$ be the Weil restriction of scalars of $X$ to $K^{p^n} \subset K$, which has dimension $p^n \cdot \dim(X)$ and comes with a natural homeomorphism $X_n(K^{p^n}_v) \cong X(K_v)$ for every place $v$. For an open $ U \subset X(K_v)$ we say that a continuous function $U \rightarrow K_v$ is $p^n$-analytic if around every point $w \in U $ it is analytic when considered as a function on $ X_n$.
\end{defi}
\begin{rmk}
Note that $X_n$ will in general not be proper, since we are restricting scalars along purely inseparable extensions. 
\end{rmk}
\begin{prop} \label{main prop evaluation map}
Let $X/K$ be a smooth variety. For any $\omega \in H^0(X, \Omega^1_X)$, the pullback map 
   \begin{align*}
     \mathrm{ev}^{\partial}_{\omega} \colon  X(K_v) &\rightarrow K_v \\
       x &\mapsto \omega_{|x}
   \end{align*} is $p$-analytic.  If moreover the image of $\omega$ in  $H^0(X, \Omega^1_{X/K})$ is nonzero (i.e., if $\omega$ is not constant) then $\mathrm{ev}^{\partial}_{\omega}$ is nonconstant around every point of $X$.
\end{prop}
\begin{proof}
Since $X$ is smooth over $K$, it is etale-locally isomorphic to $\mathbb{A}^d_K$, where $d = \dim(X)$, that is, each point of $X$ has a Zariski open neighborhood $U$ and an etale morphism $U \rightarrow \mathbb{A}^d_K$. Via standard reductions, it is then enough to prove the statement when $X$ is a Zariski open subset $X \subset \mathbb{A}^d_K$ and $\omega$ is any differential form $\omega = f_0(x_1, \cdots, x_d) dt + \sum_{i=1}^d f_i(x_1, \cdots, x_d) dx_i$ where each $f_i \in K_v[[x_1, \cdots, x_d]]$ converges on $U$ and we identify $K_v \cong \Ff((t))$. Now, for any point $u = (u_1, \cdots, u_d) \in U(K_v)$ the evaluation map is simply 
$$\mathrm{ev}^{\partial}_\omega(u) =  f_0(u_1, \cdots, u_d) dt + \sum_{i=1}^d f_i(u_1, \cdots, u_d) u_i' dt;$$
but $u \mapsto u_i'$ is $p$-analytic by the previous considerations, and since composition of power series is clearly analytic when defined, the first statement follows. 

We now prove the second part of the statement by induction on $d$. So assume $d=1$ and assume that $\mathrm{ev}^{\partial}_\omega$ is constant around some point $u \in U(K_v)$. Up to a change of coordinates we can assume that $u = 0$ and we can pick $N>0$ such that $\mathrm{ev}^{\partial}_\omega$ is constant on $t^N \mathcal{O}_v$. Now we claim that $f_0$ is constant. Choose $v \in t^N \mathcal{O}_v$. Then also $v^p \in t^N \mathcal{O}_v$. Then $\mathrm{ev}^{\partial}_\omega(v^p) = f_0(v^p) dt + f_1(v^p) d(v^p) = f_0(v^p)dt$ is independent of $v$, which implies that $f_0 \in K_v$, proving the claim. 

Again, choose $v \in t^N \mathcal{O}_v$, and compute 
$0 = \mathrm{ev}^{\partial}_\omega(v) - \mathrm{ev}^{\partial}_\omega(0) = f_1(v) v'dt$. But this can only happen if $f_1 = 0$, which contradicts the assumption of the theorem. For $d>1$, we can cut $U$ with a hyperplane $H$ such that $\omega_{| H}$ also satisfies the assumption, and then conclude by induction.
\end{proof}
We now prove analogous results for sections of $H^0(X, \Omega^1_{X}/ B_{n, X})$. The induction argument in the previous proof does not work for the general argument, so we need to replace it with two lemmas. 

\begin{lemma}\label{lemma curve evaluation}
Let $Y/K_v$ be a smooth integral curve, and let
$\omega\in H^0(Y,\Omega^1_Y)$. If $d\omega\neq0$, then, for every $n\geq1$, evaluation of $\omega$ modulo $B_{n,K_v}$ is nonconstant
on every nonempty local open.
\end{lemma}\begin{proof}
Fix a uniformizer $t$ of $K_v$, and use primes for the
$t$-derivation. Inside any given nonempty local open, choose
a coordinate disc with coordinate $s$, on which
$\omega=f_0(s)\,dt+f_1(s)\,ds$, with
$f_i(s)=\sum_{j\geq0}a_{i,j}s^j\in K_v[[s]]$
convergent on the disc.
Write $\partial_s f_i$ for the formal derivative in $s$,
and define $\partial_t f_i=\sum_{j\geq0}a_{i,j}'s^j$.
These series also converge on the disc, and
$(f_i(u))'=\partial_t f_i(u)+\partial_s f_i(u)u'$. Since
$d\omega=(\partial_t f_1-\partial_s f_0)\,dt\wedge ds\neq0$,
we can choose $u$ in the disc such that
$(\partial_s f_0-\partial_t f_1)(u)\neq0$.
Set $\alpha=\frac{t^{p^n-1}}{(\partial_s f_0-\partial_t f_1)(u)}$.
For sufficiently small $\epsilon\in K_v$, Taylor expansion
and $(u+\epsilon^{p^n}\alpha)'=u'+\epsilon^{p^n}\alpha'$
give
\[
\begin{aligned}
\omega_{|u+\epsilon^{p^n}\alpha}-\omega_{|u}
&=\epsilon^{p^n}
\bigl(
\alpha\,\partial_s f_0(u)
+\alpha u'\,\partial_s f_1(u)
+\alpha'f_1(u)
\bigr)\,dt
+O(\epsilon^{2p^n})\\
&=\epsilon^{p^n}
\bigl(t^{p^n-1}\,dt+d(f_1(u)\alpha)\bigr)
+O(\epsilon^{2p^n}).
\end{aligned}
\]
Here the remainder is a convergent series in powers of
$\epsilon^{p^n}$. The projection to $\Omega^1_{K_v}/B_{n,K_v}$ is
$K_v^{p^n}$-linear and kills exact forms. Thus, after
projection, the coefficient of $\epsilon^{p^n}$ is
$[t^{p^n-1}dt]\neq0$, since
$C^n(t^{p^n-1}dt)=dt$.
As $\epsilon^{p^n}$ ranges over a neighbourhood of zero
in $K_v^{p^n}$, uniqueness of convergent power series
shows that evaluation modulo $B_{n,K_v}$ cannot be
constant on the disc.
\end{proof}
\begin{lemma}[Restriction to a curve]\label{lemma restriction curve}
Let $X/K$ be a smooth integral variety,
$\omega\in H^0(X,\Omega^1_X)$, and
$V\subset X(K_v)$ a nonempty local open.
If $d\omega\neq0$, there is a smooth integral curve
$Y\subset X_{K_v}$ meeting $V$ such that
$d(\omega_{|Y})\neq0$.
\end{lemma}

\begin{proof}
Fix a uniformizer $t\in K$ at $v$, and use primes for
the $t$-derivation on $K_v$ as in the proof of the previous lemma. Choose $u\in V$ and an affine neighbourhood $U\subset X$
with étale coordinates $x_1,\ldots,x_d$. Put $u_i=x_i(u)$. If $d=1$ there is nothing to prove, so we may assume $d\geq2$. Write
\[
\omega=f_0\,dt+\sum_{i=1}^d f_i\,dx_i,
\qquad f_i\in\mathcal O_X(U)
\]
and we compute
\[
d\omega=
\sum_{i=1}^d
(\partial_t f_i-\partial_{x_i}f_0)\,dt\wedge dx_i
+\sum_{1\leq i<j\leq d}
(\partial_{x_i}f_j-\partial_{x_j}f_i)\,dx_i\wedge dx_j.
\]
Put $g_{0i}=\partial_t f_i-\partial_{x_i}f_0$ and
$g_{ij}=\partial_{x_i}f_j-\partial_{x_j}f_i$. Since $d\omega\neq0$, some $g_{0i}$ or $g_{ij}$ is
nonzero. If some $g_{0i}\neq0$, permute the coordinates
so that $g_{01}\neq0$. Otherwise, permute them so that
$g_{12}\neq0$. In either case, at least one of
$g_{01},g_{12},\ldots,g_{1d}$ is nonzero. Next, choose $a_2,\ldots,a_d\in K_v$, sufficiently close to
$u_2,\ldots,u_d$, such that
$a_2,\ldots,a_d,a_2',\ldots,a_d'$ are algebraically
independent over $K$. This is done as follows: choose sufficiently small $b_i,c_i\in K_v^p$,
all algebraically independent over the field generated
by $K$, the $u_i$, and the $u_i'$.
Set $a_i=u_i+b_i+tc_i$, so $a_i'=u_i'+c_i$.
Such $b_i,c_i$ exist by cardinality reasons. Now, the equations $x_2=a_2,\ldots,x_d=a_d$ define a smooth
curve in $U_{K_v}$. Since the coordinate map induces
a local homeomorphism on $K_v$-points, this curve meets
$V$ when the $a_i$ are sufficiently close to the $u_i$.
Let $Y$ be a component meeting $V$, and put $s=x_{1|Y}$.

Since $s$ is transcendental over $K_v$, the elements
$s,a_2,\ldots,a_d$ are algebraically independent over
$K$. By algebraic independence, the generic point of $Y$
maps to the generic point of $\mathbb A^d_K$.
Since $U\to\mathbb A^d_K$ is étale, it therefore maps
to the generic point of $U$, giving an embedding
$K(X)\hookrightarrow K_v(Y)$, algebraic over $K(s,a_2,\ldots,a_d)$.
Moreover, $a_2',\ldots,a_d'$ are algebraically independent
over $K(s,a_2,\ldots,a_d)$, and hence over the image
of $K(X)$. Now $dx_{i|Y}=a_i'\,dt$ for $i\geq2$, so
\[
d(\omega_{|Y})=
\left(
g_{01|Y}-\sum_{i=2}^d g_{1i|Y}\,a_i'
\right)dt\wedge ds.
\]
Now note that $g_{01|Y}-\sum_{i=2}^d g_{1i|Y}\,a_i'$ cannot be zero, due to the algebraic independence
of the $a_i'$ over the image of $K(X)$, and since at least
one of $g_{01|Y},g_{12|Y},\ldots,g_{1d|Y}$ is nonzero.
As $dt\wedge ds\neq0$, this proves the claim.
\end{proof}
We can now prove the analogue of Proposition~\ref{main prop evaluation map} in the general case.
\begin{prop} \label{main prop evaluation map 2}
Let $n \geq 1$ and fix $\omega \in H^0(X, \Omega^1_{X}/ B_{n, X})$. Then the map $\mathrm{ev}^{\partial}_{\omega} \colon X(K_v) \rightarrow \Omega^1_{K_v}/B_{n, K_v}$ is $p^{n}$-analytic. If moreover $X\cong\tilde X^{(n)}$ for a smooth $K$-variety $\tilde X$ and $\omega\neq0$ lies in the image of the morphism $$ H^0(\tilde{X}, \Omega^1_{\tilde{X}/K}/ B_{n, \tilde{X}/K}) \rightarrow H^0(X, \Omega^1_{X}/ B_{n, X})$$
from Proposition \ref{prop morphism descent}, then $\mathrm{ev}^{\partial}_{\omega}$ is not constant around any point $x \in X(K_v)$. 
\end{prop}
\begin{proof}
Locally, write $\omega=\omega'+B_{n,X}$, where $\omega'$
is a regular differential form. By
Proposition~\ref{main prop evaluation map},
$\mathrm{ev}^{\partial}_{\omega'}$ is $p$-analytic, hence
$p^n$-analytic. Since the projection
$\Omega^1_{K_v}\to\Omega^1_{K_v}/B_{n,K_v}$ is
$K_v^{p^n}$-linear, $\mathrm{ev}^{\partial}_{\omega}$
is also $p^n$-analytic. Suppose now that $\omega$ satisfies the descent hypothesis
and that evaluation is constant on a nonempty local open
$V\subset X(K_v)$. After shrinking $V$, choose an affine
open $U\subset X$ with $V\subset U(K_v)$ and a representative
$\omega'\in H^0(U,\Omega^1_U)$. Suppose that, for some $0\leq r<n$, the form $C^r(\omega')$
is defined but is not closed. By Lemma~\ref{lemma restriction curve},
there is a smooth integral curve $Y\subset U_{K_v}$
meeting $V$ such that
$d((C^r(\omega'))_{|Y})\neq0$.
Lemma~\ref{lemma curve evaluation}, applied with $n-r$, shows that
\[
y\longmapsto
(C^r(\omega'))_{|y}\pmod{B_{n-r,K_v}}
=
C^r(\omega'_{|y})\pmod{B_{n-r,K_v}}
\]
is not constant on $V\cap Y(K_v)$.
This contradicts constancy of evaluation modulo $B_{n,K_v}$,
since Cartier induces a map
$\Omega^1_{K_v}/B_{n,K_v}\to
\Omega^1_{K_v}/B_{n-r,K_v}$.
Consequently, $\omega',C(\omega'),\ldots,C^{n-1}(\omega')$
are all closed, and $C^n(\omega')$ is defined. We claim that $C^n(\omega')$ has nonzero image in
$\Omega^1_{U/K}$.
Indeed, Proposition~\ref{prop morphism descent} implies
that $\omega$ has nonzero image in
$H^0(X,\Omega^1_{X/K}/B_{n,X/K})$.
Since this quotient is locally free on the relative
Frobenius twist, its restriction to $U$ is nonzero,
and flat base change preserves this nonvanishing over
$\bar K$.
If $C^n(\omega')$ had zero relative component, then
\[
C^n(\omega'_{\bar K})
=(C^n(\omega'))_{\bar K}=0,
\]
so $\omega'_{\bar K}\in B_{n,U_{\bar K}}$, contradicting
that nonvanishing. Proposition~\ref{main prop evaluation map} now shows that
$x\mapsto(C^n(\omega'))_{|x}$ is not constant on $V$.
But this equals $C^n(\omega'_{|x})$, which is constant
because evaluation of $\omega$ modulo $B_{n,K_v}$ is
constant. This final contradiction proves the result.
\end{proof}

\subsubsection{Consequences for Brauer--Manin set and rational points} \label{sect Constant}
We now use the results of the previous section to deduce that varieties whose differential Brauer group is nontrivial have a small Brauer-Manin set, and hence sparse global points. 
\begin{prop}
Let $X/K$ be a smooth projective variety over a global field $K$ of positive characteristic. Let $\omega \in H^0(X, \Omega^1_X/ B_{n,X})$ be any section, and consider the evaluation map 
\begin{align*}
    \mathrm{ev}_{\omega}^{\partial} \colon X(\A_K) &\rightarrow \Omega^1_{\A_K} \cong \A_K \\
    (x_v) &\mapsto C^n(\omega_{| x_v}).
\end{align*}
Then, there are finitely many $\sigma_1, \cdots, \sigma_N \in \Omega^1_K \subset \Omega^1_{\A_K}$ such that 
$$X(\A_K)^{\Br(X)} \subset \bigcup_{i=1}^N (\mathrm{ev}_{\omega}^{\partial})^{-1} \{ \sigma_i \}.$$
\end{prop}
\begin{proof}
The evaluation map $\mathrm{ev}_{\omega}^{\partial}$ is continuous and $X(\A_K)^{\Br(X)} \subset X(\A_K)$ is closed and hence compact, because $X(\A_K)$ is compact as $X/K$ is proper by assumption. Moreover, by Theorem \ref{theorem 1 intro}, we have the inclusion $$X(\A_K)^{\Br(X)} \subset \bigcup_{\sigma \in \Omega^1_K} (\mathrm{ev}_{\omega}^{\partial})^{-1} \{ \sigma \}.$$ Since $\Omega^1_K \subset \Omega^1_{\A_K}$ is discrete, and since the image $\mathrm{ev}_{\omega}^{\partial}(X(\A_K)^{\Br(X)}) \subset \Omega^1_K$ is compact, this image has to be finite, which concludes the proof. \end{proof}
In particular, we have the following corollary. 
\begin{cor}
Let $X/K$ be a smooth projective variety over a global field of positive characteristic. Assume that $H^0(\bar{X}, \Omega^1_{\bar{X}}/ B_{n, \bar{X}}) \neq 0$ for some $n>0$. Assume that $X$ descends to $K^{p^n}$. Then, for every place $v$ of $K$, the $v$-adic closure of $X(K)$ in $X(K_v)$ is contained in finitely many proper $p^n$-analytic subvarieties (and, in particular, has $v$-adic measure zero). 
\end{cor}
\begin{proof}
By Proposition \ref{main prop evaluation map 2}, we can choose a global section $\omega \in H^0(X, \Omega^1_X/ B_{n,X})$ whose $v$-adic evaluation map $\mathrm{ev}_{\omega, v}^{\partial} \colon X(K_v) \rightarrow \Omega^1_{K_v} \cong K_v$ is nonconstant around any point of $x$. By the previous proposition, the $v$-adic closure of $X(K)$ in $X(K_v)$ is contained in a finite union of subsets of the form $(\mathrm{ev}_{\omega, v}^{\partial})^{-1} \{ \sigma \} $, where $\sigma \in \Omega^1_K \subset \Omega^1_{K_v}$. This concludes the proof.  
\end{proof}
\subsection{Varieties with many differential forms}
In this last part of the article we study the Brauer--Manin set of varieties with many differential forms. 
\subsubsection{Line bundles}
Here we prove Theorem \ref{Theorem Intro LinBun}. Now, if $\mathcal{L} \subset \Omega^1_X$ is an invertible subsheaf then either $\mathcal{L} = \Omega^1_K \otimes_K \Oo_X$, in which case $\psi_{\mathcal{L}}$ is trivial, or $\mathcal{L} \cap \Omega^1_K \otimes_K \Oo_X = 0$ and $\mathcal{L}$ is a lift of some $\mathcal{L}' \subset \Omega^1_{X/K}$.
\begin{proof}[Proof of Theorem \ref{Theorem Intro LinBun}]
Let $(x_v)$ be as in the statement and consider another local component $x_w$. By Theorem \ref{theorem 1 intro} there is a form $\tilde{\omega} \in \Omega^1_K \setminus 0$ such that $\omega_{| x_w}= \tilde{\omega}$ for every place $w$. So in particular $\omega_{|x_w}\neq 0$ as well. Since $\omega$ is necessarily nonzero, we can extend it to a $K$-basis $\omega = \omega_0, \omega_1, \cdots, \omega_N$ of $H^0(X, \mathcal{L})$. Again by Theorem \ref{theorem 1 intro} there are $\tilde{\omega}_i \in \Omega^1_K$ such that $\omega_{i, | x_w} = \tilde{\omega}_i$ for every $i$ and every $w$. Now pick an affine Zariski open $U \subset X$ such that $x_w \in U$. Our first claim is that $\psi_{\mathcal{L}}(x_w)$ is global. Up to restricting $U$, we can assume that $\omega$ never vanishes on $U$. In fact, $x_v \notin Z(\omega)$ for otherwise $\omega|_{x_v} = 0$. We can thus find regular functions $g_{k} \in \Oo(U)$ such that we have $\omega_{k} = g_{k} \omega$ over $U$ for every $k \geq 0$. But then $\tilde{\omega}_k = g_{k}(x_w) \tilde{\omega}$, from which we conclude that $g_k (x_w) \in K$ for every index $k$, and hence $\psi_{\mathcal{L}}(x_w) = [1 \colon g_1 (x_w) \colon g_2 (x_w) \colon \cdots \colon g_N (x_w)] \in \Pp^N(K)$. Finally, pick another component $x_{w'}$. By the same reasoning as before, we can assume that $x_w, x_{w'} \in U$ and $\omega$ never vanishes on $U$. The same computation then shows $\psi_{\mathcal{L}} (x_{w'}) = \psi_{\mathcal{L}} (x_w)$, which concludes the proof.
\end{proof}
The following corollary also appears in \cite{MR4448590}:
\begin{cor} \label{cor prod curves}
Let $C/K$ be a smooth projective curve which descends to $K^p$. Assume that $\omega_{C/K}$ is very ample (i.e., if $g(C) \geq 2$ and $C$ is not hyperelliptic) and let $(x_v) \in C(\A_K)^{\Br(C)[p]}$. If there is some local component $x_v$ such that $x_v \notin C(K_v)^p$ then $(x_v) \in C(K)$. 
\end{cor}
\begin{proof}
We only need to show that if $x_v \notin C(K_v)^p$ then there is some $\omega \in H^0(C, \Omega^1_{C/K}) \subset H^0(C, \Omega^1_{C})$ such that $\omega_{| x_{v}} \neq 0$. But  $H^0(C, \Omega^1_{C/K}) $ generates $\Omega^1_{C/K}$ and we use Lemma \ref{lemma vertical} from the next section to conclude. 
\end{proof}

\subsubsection{Weakly separable points and globally generated differentials} \label{sec good points}
In this section we aim to generalize the previous corollary to higher-dimensional varieties. We assume that $X$ descends along the Frobenius of $K$ and we fix a model $\tilde{X}/K$. Denote by $F \colon \tilde{X} \rightarrow X$ the relative Frobenius of $\tilde{X}$ and recall that for any $K$-algebra $A$ we denote $X(A)^p \coloneqq F(\tilde{X}(A)) \subset X(A)$. Note that if $L \subset K^s$ we have $X(L) \cap X(K^s)^p  = X(L)^p$, so there is a natural inclusion $X(L) \setminus X(L)^p \subset X(K^s) \setminus X(K^s)^p$. We call a local section of $\Omega^1_X$ \textit{vertical} if it has the form $(0, \omega) \in \Omega^1_K \otimes \Oo_X \oplus \Omega^1_{X/K}$, where the splitting is induced by the chosen Frobenius descent. 
\begin{lemma} \label{lemma vertical}
Let $L/K$ be any field extension and let $a \in X(L)$. Then $a \in X(L)^p$ if and only if $\omega_{|a} = 0 \in \Omega^1_L$ for every local vertical differential. 
\end{lemma}
\begin{proof}
Locally around $a$ we can write $X = \Spec(A)$ with $A = K[x_1, \cdots, x_n]/I$ where $I$ is generated by a regular sequence $f_1, \cdots, f_k$ of polynomials in $K^p[x_1, \cdots, x_n]$. A local generating set for the vertical differentials is then given by the various $dx_i$. The point $a$ gives elements $x_i(a) \in L$ and in these coordinates we have that $a \in X(L)^p$ if and only if $x_i(a) \in L^p$ for every $i$. Finally, this is equivalent to ${(dx_i)}_{|a} = 0$ for every $i$. 
\end{proof}
Recall now that a field extension in positive characteristic $\kappa \subset L$ is separable if and only if $\Omega^1_\kappa \otimes_\kappa L \rightarrow \Omega^1_L$ is injective (see \cite[\href{https://stacks.math.columbia.edu/tag/031X}{Tag 031X}]{stacks-project}). 
\begin{defi}
A field extension $\kappa \subset L$ is weakly separable if $\Omega^1_\kappa \rightarrow \Omega^1_L$ is injective. If $X/K$ is a variety and $K \subset L$ is any field, we say that $a \in X(L)$ is weakly separable if its image $x \in X$ satisfies 
\begin{enumerate}[label=(\roman*)]
    \item $K \subset \kappa(x)$ is separable;
    \item $\kappa(x) \subset L$ is weakly separable.
\end{enumerate}
Here $\kappa(x)$ is the residue field of $x \in X$. We denote by $X^{\mathrm{ws}}(L) \subset X(L)$ the subset of weakly separable $L$-points. 
\end{defi}
Note that in our situation point (1) is automatically satisfied:
\begin{lemma}
Let $K$ be a global field and let $K_v$ be the completion of $K$ at the place $v$. For any $a \in X(K_v)$ let $x \in X$ be its image. Then $K \subset \kappa(x)$ is separable. 
\end{lemma}
\begin{proof}
In fact $K \subset K_v$ is separable since $\Omega^1_K \otimes_K K_v \cong \Omega^1_{K_v}$. But then the map $\Omega^1_K \otimes_K \kappa(x) \rightarrow \Omega^1_{\kappa(x)}$ is necessarily injective. 
\end{proof}
On the other hand, if, for example, $\deg[\kappa(x) \colon \kappa(x)^p] \geq p^2$ (e.g., if $x$ is the generic point of a subvariety of dimension larger than one) then the inclusion $\kappa(x) \subset K_v$ can never be separable by dimension reasons, since $\dim_{K_v}(\Omega^1_{\kappa(x)} \otimes_{\kappa(x)} K_v) > \dim_{K_v}(\Omega^1_{K_v}) = 1$. On the other hand, we show that it is weakly separable in the very general case:\begin{thm}\label{Theorem density good}
Let $X/K$ be a smooth variety and let $x\in X$ satisfy
that $K\subset\kappa(x)$ is separable.
Let $X(K_v,x)\subset X(K_v)$ be the subset of local
points with image $x$.
If $X(K_v,x)\neq\emptyset$, then
$X^{\mathrm{ws}}(K_v,x)
\coloneqq X^{\mathrm{ws}}(K_v)\cap X(K_v,x)$
is the complement in $X(K_v,x)$ of countably many
proper $p$-analytic subsets. In particular, it is dense.
\end{thm}

\begin{proof}
Let $Z=\overline{\{x\}}$ with its reduced structure.
Since $K\subset\kappa(x)$ is separable, choose a nonempty
smooth affine open $U\subset Z$.
Every point of $X(K_v,x)$ factors through $U$, and $x$
is the generic point of $U$.

Fix a uniformizer $t\in K$ at $v$.
For every $\omega\in H^0(U,\Omega^1_U)$ with nonzero
relative component, Proposition~\ref{main prop evaluation map}
shows that $\mathrm{ev}^{\partial}_\omega$ is $p$-analytic
and nowhere locally constant. Its zero locus is therefore
a closed $p$-analytic subset with empty interior.

A nonzero form with zero relative component is
$\omega=f\,dt$ for some nonzero $f\in\mathcal O_U(U)$.
For $a\in X(K_v,x)$ we have $f(a)\neq0$, since $a$ maps
to the generic point. Thus $\omega_{|a}\neq0$.

Moreover,
\[
\Omega^1_{\kappa(x)}
=
H^0(U,\Omega^1_U)\otimes_{\mathcal O_U(U)}\kappa(x).
\]
Consequently, every rational differential has a nonzero
regular multiple on $U$. Since nonzero regular functions
evaluate nontrivially at $a\in X(K_v,x)$, weak separability
is equivalent to nonvanishing of every nonzero regular
differential at $a$. Hence
\[
X^{\mathrm{ws}}(K_v,x)
=
X(K_v,x)\setminus
\bigcup_{\omega}
(\mathrm{ev}^{\partial}_\omega)^{-1}(0),
\]
where the union ranges over all regular absolute forms
on $U$ with nonzero relative component.

Since $K$ is countable, this union is countable.
The points of $U(K_v)$ not mapping to $x$ also form
a countable union of closed nowhere dense subsets:
\[
U(K_v)\setminus X(K_v,x)
=
\bigcup_{0\neq f\in\mathcal O_U(U)}
\{a\in U(K_v):f(a)=0\}.
\]
Baire's theorem now shows that $X^{\mathrm{ws}}(K_v,x)$
is dense in $U(K_v)$, hence in $X(K_v,x)$.
Each exceptional zero locus also has measure zero,
so their union has measure zero by countable additivity.
\end{proof}
\begin{defi} \label{def good points}
We say that $a \in X(L)$ is good if it is weakly separable and $a \notin X(L)^p$. We denote by $X^{\mathrm{good}}(L) \subset X(L)$ the subset of good points. 
\end{defi}
The following facts are easily verified:
\begin{enumerate}[label=(\roman*)]
    \item $X^{\mathrm{good}}(K^s) = X(K^s) \setminus X(K^s)^p$;
\item
$X(L)\cap X^{\mathrm{good}}(L')
=X^{\mathrm{good}}(L)$
for any finite separable extension $L'/L$;

\item
$X^{\mathrm{good}}(K_v)$ contains the complement of
a countable union of proper $p$-analytic subvarieties
of $X(K_v)$.
\end{enumerate}
We finally define $X^{\mathrm{good}}(\A_K) \subset X(\A_K)$ as the subset of adelic points all whose components are good. 

Assume now that $\Omega^1_{X/K}$ is globally generated. We study the Brauer--Manin obstruction on $X^{\mathrm{good}}(\A_K)$. Let $S \coloneqq H^0(X, \Omega^1_{X/K}) \subset H^0(X, \Omega^1_X)$ be the subspace of vertical differentials, and put $\mathcal{S} = S \otimes \Oo_X$, so we have an exact sequence of locally free sheaves 
\begin{equation} \label{eq whome omega}
    0 \rightarrow \mathcal K \rightarrow \mathcal S \rightarrow \Omega^1_{X/K} \rightarrow 0,
\end{equation}
where $\mathcal K$ is defined by the sequence. If $s = \dim S$, this yields a map over $K$ \begin{equation} \label{Definition psi}
    \psi \colon X \rightarrow \mathrm{Grass(S, s-d)},
\end{equation}
where $d = \dim X$. By definition, if $L/K$ is any field extension and $a \in X(L)$ then $$\psi(a)  = \{ \omega \in S_L  \colon a \in Z(\omega) \} $$
where $S_L = S \otimes_K L$ and $Z(\omega) \subset X_L$ is the vanishing locus of $\omega$. In particular, if $\omega \in \psi(a)$, then $\omega_{| a} = 0 \in \Omega^1_L$ by functoriality of pullback. Now, any $x \in X^{\mathrm{good}}(K^s)$ yields a $K^s$-linear map
\begin{align}\label{aux map}
    \theta_x \colon S_{K^s} & \rightarrow \Omega^1_{K^s}  \\
    \omega & \mapsto \omega_{|x}.
    \end{align}
\begin{prop} 
If $\Omega^1_{X/K}$ is globally generated then $\theta_x$ is surjective.
\end{prop}
\begin{proof}
Pick an affine chart $x \in U \subset X$ with $U = \Spec(K[x_1, \cdots, x_n]/(f_1, \cdots, f_k))$ where each $f_i$ has coefficients in $K^p$. Since $x$ is not a $p$-power we may assume $x_1(x) \notin (K^s)^p$. By global generation there are sections $\omega_1, \cdots, \omega_n \in S$ and local functions $f_i$ defined around $x$ such that $d x_1 = \sum_i f_i \omega_i$, from which it follows that $0 \neq \sum f_i(x) \omega_{i| x}$ and hence that some $\omega_{i| x}$ must be nonzero. Since $\Omega^1_{K^s}$ is one-dimensional, this concludes the proof.
\end{proof}
We now consider the map 
\begin{align}\label{aux map 2}
    \theta \colon X^{\mathrm{good}}(K^s) & \rightarrow \Pp(S^\vee)(K^s)   \\
    x & \mapsto \ker(\theta_x) \nonumber 
    \end{align}
Note that the two maps $\psi$ and $\theta$ coincide when $X$ is a curve. The fact that $\psi \neq \theta$ for $\dim(X) \geq 2$ is the reason why this case is more complicated and why we needed the notion of good points to begin with.
\begin{thm} \label{Theorem glob gen}
Assume that $X$ descends along $F_K$ and that $\Omega^1_{X/K}$ is globally generated. Let $(x_v) \in X^{\mathrm{good}}(\A_K)^{\Br_0^{\partial}(X)[p]}$. Then each local component $x_v$ belongs to $X^{\mathrm{good}}(K^s)$ and any two components satisfy $\theta(x_v) = \theta(x_w)$. If moreover $\theta$ is injective then $(x_v) \in X^{\mathrm{good}}(K) = X(K) \setminus X(K)^p$.
\end{thm}
\begin{proof}
Take a local component $x_v$ with image $z_v\in X$.
For every $\omega\in S$, Theorem~\ref{theorem 1 intro}
gives $\widetilde\omega\in\Omega^1_K$ such that
$\omega_{|x_v}=\widetilde\omega$.
Since $x_v$ is good, the map
$\Omega^1_{\kappa(z_v)}\to\Omega^1_{K_v}$ is injective.
Consequently, $\omega_{|z_v}$ belongs to the image of
$\Omega^1_K\to\Omega^1_{\kappa(z_v)}$, so its image in
$\Omega^1_{\kappa(z_v)/K}$ is zero.
By global generation, these images generate
$\Omega^1_{\kappa(z_v)/K}$, so this space is zero.
Since $\kappa(z_v)/K$ is finitely generated and separable,
it is finite separable. This proves the first assertion. The second assertion follows easily: pick another component $x_w$ with image $z_w \in X^{\mathrm{good}}(K^s)$; then Theorem \ref{theorem 1 intro} implies that $\theta_{x_v} = \theta_{x_w}$ too. Finally, if $\theta$ is injective then every $x_v$ must be supported on the same point $x \in X(K^s)$. But a simple application of Chebotarev shows that $x \in X(K)$ necessarily, because there are no algebraic extensions of $K$ between $K$ and $\A_K$. 
\end{proof}
 Finally, we prove
\begin{prop}
Assume that $\Omega^1_{X/K}$ is $1$-jet ample. Then $\theta$ is injective. In particular $X^{\mathrm{good}}(\mathbb{A}_K)^{\Br(X)[p]} = X^{\mathrm{good}}(K)$ in this case. 
\end{prop}
We say that a vector bundle $E$ is $k$-jet ample if, for every choice of
distinct $x_i\in X(\bar K)$ and positive integers $a_i$ with $\sum_i a_i=k+1$,
evaluation of global sections of $E_{\bar K}$ onto
$\bigoplus_i H^0(X_{\bar K},E_{\bar K}\otimes
\mathcal O_{X_{\bar K}}/\mathfrak m_{x_i}^{a_i})$ is surjective.
\begin{proof}Let $x,y\in X^{\mathrm{good}}(K^s)$ be distinct.
Surjectivity of simultaneous evaluation at $x$ and $y$ gives
$\dim(\psi(x)\cap\psi(y))=s-2d$, and hence
$\dim(\psi(x)+\psi(y))=s$.
Hence $\psi(x)+\psi(y)=S_{K^s}$.
Since $\psi(x)\subset\ker(\theta_x)$ and
$\psi(y)\subset\ker(\theta_y)$, equality
$\theta(x)=\theta(y)$ would put both subspaces in the
same hyperplane, contradicting their sum being $S_{K^s}$.
Thus $\theta$ is injective.
\end{proof}
We remark that the injectivity of $\theta$ is most likely much weaker than $1$-jet ampleness.

\section{Acknowledgment}
The author was partially supported by ERC Starting Grant No. 804334. This project originated from a question posed by Alexei Skorobogatov during the BIRS–CMI workshop \textit{New Directions in Rational Points} in Chennai. The author gratefully acknowledges the hospitality of BIRS and CMI, as well as Alexei Skorobogatov’s curiosity and sustained support. The author also thanks Felipe Voloch for comments on the article, Fabio Bernasconi for drawing his attention to Bogomolov–Sommese vanishing and other geometric phenomena in positive characteristic, and Jefferson Baudin, Alapan Mukhopadhyay, and Margherita Pagano for useful discussions.
\section*{Declaration of AI assistance}
ChatGPT (OpenAI) assisted the author with the proof of
Proposition \ref{main prop evaluation map 2} and identified a minor error in Theorem \ref{Theorem density good} which it also helped to correct.
\bibliographystyle{alpha}
\bibliography{bibliography}

\end{document}